\documentclass[11pt,a4paper]{article}
\usepackage{latexsym,amsmath,amssymb,times}
\usepackage[latin1]{inputenc}   
\usepackage[all]{xy}
\newtheorem{theorem}{Theorem}[section]
\newtheorem{proposition}[theorem]{Proposition}
\newtheorem{corollary}[theorem]{Corollary}
\newtheorem{lemma}[theorem]{Lemma}
\newtheorem{example}{Example}[section] 
\newenvironment{proof}{\noindent \bf Proof
\rm}{\hspace*{\fill} $\Box $ \vskip7pt}

\title{Hopf algebra structure of incidence algebras}
\author{Dieter Denneberg, Universität Bremen \\{\small
denneberg@math.uni-bremen.de}}

\newcommand{\N}{\ensuremath{\mathbb{N}}}
\newcommand{\Prim}{\ensuremath{\mathbb{P}}}
\newcommand{\R}{\ensuremath{\mathbb{R}}}

\newcommand{\C}{\ensuremath{\mathbb{C}}}
\newcommand{\Quat}{\ensuremath{\mathbb{H}}}

\newcommand{\Zeta}{{\rm Z}}

\newcommand{\Mu}{{\rm M}}
\newcommand{\id}{\mbox{\rm id}}
\newcommand{\image}{\mbox{\rm image}}

\newcommand{\intv}{\mbox{\,\rm Intv}\,}
\newcommand{\Hom}{\mbox{\,\rm Hom}\,}
\newcommand{\Bilin}{\mbox{\,\rm Bilin}\,}

\begin{document}

\newcounter{reg1}

\maketitle

\begin{abstract}
The incidence algebra of a partially ordered set (poset) supports in a natural way also a coalgebra structure, so that it becomes a m-weak bialgebra even a m-weak Hopf algebra with Möbius function as antipode. Here m-weak means that multiplication and comultiplication are not required to be coalgebra- or algebra-morphisms, respectively. A rich theory is obtained in computing modulo an equivalence relation on the set of intervals in the poset.
\end{abstract}

\noindent

\section{Introduction}

Usually, the theory of algebras, coalgebras and bialgebras is formulated without reference to bases of the linear spaces supporting these structures. But in practice, some algebraic structures, especially comultiplication, is given by referring to specific bases. Incidence algebras of posets possess a canonical base, namely the set of intervals of the poset.

Many interesting incidence algebras are no bialgebras, multiplication and comultiplication are not compatible. Nevertheless they possess an antipode. So we introduce the issues of {\it m-weak bialgebra}\footnote{The {\it m} in {\it m-weak} refers to {\it m}ultiplication. In the literature (see \cite{Boehm}, \cite{Nill}) the term {\it weak bialgebra} is used for relaxing the conditions on unit and counit.
}
and {\it m-weak Hopf algebra} and look what remains valid from the theory of bialgebras and Hopf algebras.

It is supposed that the reader is familiar with tensor products of linear spaces. The elements of coalgebras and bialgebras are repeated in a concise form. In order to ease the access we avoid Sweedler's notation (\cite{Sweedler}) and confine the theory to linear spaces. Most results generalize to modules over a commutative ring. In the next section, Section 2, we fix our notations and present some generalizations and examples needed later on. Section 3 is devoted to the convolution algebra of a coalgebra-algebra pair. The more involved dual problem, to define a coalgebra for an algebra-coalgebra pair is treated in Section 4.
Except for Example \ref{AlgDual} Section 4 can be skipped at first reading. Section 5 contains the basics of m-weak bialgebras and m-weak Hopf algebras.

Then we turn to incidence algebras and their m-weak Hopf algebra structure. In Section 6 the interval algebra and interval coalgebra of a poset are defined w.r.t.\ a bialgebra compatible equivalence relation, a strengthening of Schmitt's issue of order compatibility (\cite{Schmitt}\footnote{Notice, that {\it incidence Hopf algebras} are no incidence algebras in the usual sense: the multiplication there is not convolution. But comultiplication is the same as in our paper.}). Section 7 gives the usual definition of incidence algebra of a poset and shows that in many cases it can be embedded in the convolution algebra of the pair interval coalgebra and interval algebra of the poset. For finite posets, the Möbius function is the antipode. Section 8 presents classical examples of commutative incidence algebras of infinite posets, which are (m-weak) Hopf algebras. Final Section 9 treats morphisms of incidence algebras.

The present paper is based on notes of my lecture "Hopf Algebren und Inzidenz\-algebren" in winter 2009/2010. The author expresses his thanks to Gleb Koshevoy and Hans Eberhard Porst for helpful discussions and hints.

\section{Algebras and Coalgebras}

Throughout let $K$ be a field. Applications between $K$-linear spaces are supposed to be linear. They are uniquely defined by their values on a base. We define linear applications mostly by their values on a specific base.

All $K$-algebras are supposed to be unitary. Then a $K${\bf -algebra} ${\cal A}$ is a $K$-linear space $A$ with multiplication $\cdot$ and unit $1=1_A$, in the classical notation a triple ${\cal A}=(A,\cdot,1)$ or, in the tensor product notation, a triple $(A,\triangledown,\eta)$ with linear applications
$$
\triangledown: A\otimes A \rightarrow A\;, \qquad \eta: K\rightarrow A\;,
$$
such that associativity and unitarity hold,
\begin{eqnarray*}
\label{ass}
\triangledown\circ (\triangledown\otimes\id_A) &=& \triangledown\circ(\id_A\otimes\triangledown) , \\
\label{unitarity}
 \triangledown\circ(\eta\otimes\id_A) &=& \id_A \;=\;   \triangledown\circ(\id_A\otimes\eta) \;.
\end{eqnarray*}

Reversing the arrows a {\bf coalgebra} is defined as a triple $\;{\cal C}=(C,\vartriangle,\epsilon)$ with linear applications called {\bf comultiplication} and {\bf counit}
$$
\vartriangle: C \rightarrow C\otimes C\;, \qquad \epsilon: C\rightarrow K\;,
$$
such that coassociativity (\ref{coass}) and counitarity (\ref{counitarity}) hold,
\begin{eqnarray}
\label{coass}
(\vartriangle\otimes\,\id_C) \,\circ \vartriangle &=& (\id_C\,\otimes\vartriangle) \, \circ \vartriangle\;, \\
\label{counitarity}
(\epsilon\otimes\id_C) \,\circ \vartriangle &=& \id_C \;=\;  (\id_C\otimes\epsilon) \,\circ \vartriangle\;.
\end{eqnarray}
By  (\ref{counitarity}) comultiplication is an injective linear mapping. This fact is dual to the fact that $\triangledown$ is surjective. first examples for coalgebras are collected at the end of this section.

An algebra is commutative iff $\triangledown\circ t=\triangledown$, where for a linear space $L$ the {\bf twist mapping} $t$ is defined as
\begin{eqnarray*}
t:L\otimes L \rightarrow L\otimes L\;,\qquad t(a\otimes b)= b\otimes a\,.
\end{eqnarray*}
$t$ is a linear isomorphism. A coalgebra is called {\bf cocommutative} if $\;t\,\circ \vartriangle=\vartriangle$. Generally the {\bf opposite algebra} $\;{\cal A}^{op}:=(A,\triangledown^{op},\eta)$ of the algebra $\,{\cal A}$ is defined by
$\triangledown^{op}:=\triangledown\circ t$; similarly the {\bf opposite coalgebra}$\;{\cal C}^{op}:=(C,\vartriangle^{op},\epsilon)$ of the coalgebra $\;{\cal C}$ by $\vartriangle^{op}:=\;t\,\circ \vartriangle$.

A linear application $\alpha: A \rightarrow A'$ between algebras is an algebra morphism iff
\begin{eqnarray*}
\triangledown'\circ (\alpha\otimes\alpha) \;=\; \alpha \circ \triangledown\,, \qquad
\eta' \;=\; \alpha \circ \eta\,.
\end{eqnarray*}
A linear application $\gamma: C \rightarrow C'$ between coalgebras is, by definition,  a {\bf coalgebra morphism} if
\begin{eqnarray*}
(\gamma\otimes\gamma)\,\circ \vartriangle \;=\; \vartriangle' \circ \gamma\,, \qquad \epsilon \;=\; \epsilon'\circ\gamma \,.
\end{eqnarray*}

The tensor product $(A\otimes A',\overline{\triangledown},\overline{\eta})$ of two algebras $(A,\triangledown,\eta)$ and $(A',\triangledown',\eta')$ is again an algebra with multiplication
$$
\overline{\triangledown}:= (\triangledown\otimes \triangledown') \circ T
$$
and unit
$$
\overline{\eta}:= \eta\otimes \eta'
$$
Generally, $\,T=T_{L,L'}$ denotes the {\bf capital twist mapping} for the linear spaces $L$, $L'$, defined by
\begin{eqnarray*}
T:
\begin{array}{ccc}
(L &\otimes& L') \\
 &\otimes&  \\
(L &\otimes& L')
\end{array}
\rightarrow
\left(
\begin{array}{ccc}
L \\
\otimes \\
L
\end{array} \right)
\otimes
\left(
\begin{array}{ccc}
L' \\
\otimes   \\
L'
\end{array} \right) \;, \quad
\begin{array}{lll}
(a &\otimes& a') \\
 &\otimes&  \\
(b &\otimes& b')
\end{array}
\mapsto
\left(
\begin{array}{cc}
a \\
\otimes \\
b
\end{array} \right)
\otimes
\left(
\begin{array}{cc}
a' \\
\otimes   \\
b'
\end{array} \right) \;.
\end{eqnarray*}
In the sequel we will omit the brackets. Clearly $T$ is a linear isomorphism and $T\circ T=\id$.

The tensor product $(C\otimes C',\underline{\vartriangle},\underline{\epsilon})$ of two coalgebras $(C,\vartriangle,\epsilon)$ and $(C',\vartriangle',\epsilon')$ is again a coalgebra with comultiplication
$$
\underline{\vartriangle}:= T\circ(\vartriangle\otimes \vartriangle')
$$
and counit
$$
\underline{\epsilon}:= \epsilon\otimes \epsilon'
$$

Overlining and underlining is compatible with passing to the opposite,
\begin{eqnarray}
\label{querop}
(\overline{\triangledown})^{op} = \overline{\triangledown^{op}}\,,\quad
(\underline{\vartriangle})^{op} = \underline{\vartriangle^{op}}\,.
\end{eqnarray}
The proof is straightforward using
\begin{eqnarray*}
T\circ t = (t\otimes t) \circ T \,,\quad t\circ T = T\circ (t\otimes t)\,.
\end{eqnarray*}

\begin{proposition}
\label{TensMorph}
\begin{enumerate}
\item
The tensor product of algebra or coalgebra morphisms is again an algebra or coalgebra morphism, respectively.
\item
For an algebra (coalgebra) the twist mapping $t$ is an algebra (coalgebra) morphism.
\item
For algebras (coalgebras) the capital twist mapping $T$ is an algebra (coalgebra) morphism.
\end{enumerate}
\end{proposition}

\begin{proof}
We prove the algebra cases. The coalgebra cases prove dually.

1. Let $\alpha_i:A_i\rightarrow A_i'$ be algebra morphisms and $\alpha:=\alpha_1\otimes\alpha_2:A_1\otimes A_2\rightarrow A_1'\otimes A_2'$. We know $\eta_i' \;=\; \alpha_i \circ \eta_i$ and $\triangledown_i'\circ (\alpha_i\otimes\alpha_i) \;=\; \alpha_i \circ \triangledown_i$, whence $\eta_1'\otimes\eta_2' \;=\; (\alpha_1\otimes\alpha_2) \circ (\eta_1\otimes\eta_2)$ and $(\triangledown_i'\otimes\triangledown_2')\circ ((\alpha_1\otimes\alpha_1)\otimes(\alpha_2\otimes\alpha_2)) \;=\; (\alpha_1\otimes\alpha_2) \circ (\triangledown_1\otimes\triangledown_2)$. On the other hand $T\circ\alpha\otimes\alpha= ((\alpha_1\otimes\alpha_1)\otimes(\alpha_2\otimes\alpha_2)) \circ T$ so that $\triangledown'\circ (\alpha\otimes\alpha) \;=\; \alpha \circ \triangledown$, where $\triangledown:=(\triangledown_1\otimes\triangledown_2)\circ T$ and $\triangledown':=(\triangledown_1'\otimes\triangledown_2')\circ T$ denote the multiplications of $A_1\otimes A_2$ and $A_1'\otimes A_2'$, respectively.

2. The proof is straightforward.

3. Let $(A,\triangledown,\eta)$ and $(A',\triangledown',\eta')$ be algebras. We have to show that $T_{A,A'}:(\overline{A}\otimes\overline{A},(\overline{\triangledown}\otimes \overline{\triangledown})\circ T_{\overline{A},\overline{A}}\,,\overline{\eta}\otimes\overline{\eta}) \rightarrow (A_2\otimes A_2',(\triangledown_2\otimes \triangledown_2')\circ T_{A_2,A_2'}\,,\eta_2\otimes\eta_2')$ is an algebra morphism.
Here we have used the following algebras $(\overline{A},\overline{\triangledown},\overline{\eta}):=(A\otimes A',(\triangledown\otimes\triangledown')\circ T_{A,A'},\eta\otimes\eta')\,$,
$\;(A_2,\triangledown_2,\eta_2):=(A\otimes A,(\triangledown\otimes\triangledown)\circ T_{A,A},\eta\otimes\eta)$ and similarly $\;(A_2',\triangledown_2',\eta_2')$.

Clearly $T_{A,A'}$ respects the unit. The proof of the other condition
$$
((\triangledown_2\otimes \triangledown_2')\circ T_{A_2,A_2'})\circ (T_{A,A'}\otimes T_{A,A'})  \;=\; T_{A,A'}\circ ((\overline{\triangledown}\otimes \overline{\triangledown})\circ T_{\overline{A},\overline{A}})
$$
is straightforward but tedious. It is left to the reader.
\end{proof}

The Eckmann-Hilton argument is needed to decide, under which conditions $\triangledown$ is an algebra morphism.

\begin{theorem}
\label{EckmannHilton}
Let ${\cal A}=(A,\triangledown,\eta)$ and ${\cal A'}=(A,\triangledown',\eta')$ be $K$-algebras on the same linear space $A$. If $\triangledown':(A\otimes A,\overline{\triangledown},\overline{\eta}) \rightarrow (A,\triangledown,\eta)$ is an algebra morphism, then $\triangledown'=\triangledown$, $\eta'=\eta$ and ${\cal A}$ is commutative.
\end{theorem}
In fact this folk theorem holds for monoids (e.g.\ \cite{Porst_Vorl} 3.1) since the linear structure of $A$ is not needed in the proof. Even more, associativity is not needed, too. \\

\begin{proof}
First we show $\eta'=\eta$, i.e.\ $1_A'=1_A$ where $1_A:=\eta(1)$,  $1_A':=\eta'(1)$.
\begin{eqnarray*}
1_A \;=\; \triangledown(1_A\otimes 1_A)
&=& \triangledown(\triangledown'(1_A'\otimes 1_A)\otimes \triangledown'(1_A\otimes 1_A') \\
&=& \triangledown\circ(\triangledown'\otimes \triangledown')(1_A'\otimes 1_A) \otimes (1_A\otimes 1_A') \\
&=& \triangledown'\circ(\triangledown\otimes \triangledown)\circ T((1_A'\otimes 1_A) \otimes (1_A\otimes 1_A')) \\
&=& \triangledown'\circ(\triangledown\otimes \triangledown)(1_A'\otimes 1_A) \otimes (1_A\otimes 1_A') \\
&=& \triangledown'(1_A'\otimes 1_A') \;=\; 1_A'\,.
\end{eqnarray*}
In the second equation we used unitarity of $\triangledown'$ and in the fourth that $\triangledown'$ is an algebra morphism w.r.t.\ ${\cal A}$. The fifth equation holds since the capital twist $T$ does not change the argument.

Next we show $\triangledown'\circ t=\triangledown$ and then $\triangledown'=\triangledown$, which completes the proof. Within the diagram
\begin{center}
\begin{minipage}{6cm}
\xymatrix@=4em{
{\begin{array}{c}
A \\ \otimes \\ A
\end{array}}
 \ar[r]^{\sim}\ar[d]^t &
{\begin{array}{c}
K \otimes A \\
 \otimes  \\
A \otimes K
\end{array}}
\ar[r]^{\scriptsize{\begin{array}{c}
\eta \otimes \id \\
 \otimes  \\
\id \otimes \eta
\end{array}}}
\ar[d]^{T} &
{\begin{array}{c}
A \otimes A \\
 \otimes  \\
A \otimes A
\end{array}}
\ar[r]^{\scriptsize{\begin{array}{l}
\triangledown' \\ \otimes \\ \triangledown'
\end{array}}}
\ar[d]^{T} \ar[dr]^{\overline{\triangledown}}
 & {\begin{array}{c}
A \\ \otimes \\ A
\end{array}}
 \ar[dr]^{\triangledown}
 \\
{A\otimes A} \ar[r]^{\sim} &
{\begin{array}{c}
K \quad\; A \\
\otimes\otimes\otimes \\
A \quad\; K
\end{array}}
\ar[r]^{\scriptsize{\begin{array}{c}
\eta \quad\; \id \\
\otimes\otimes\otimes  \\
\id \quad\, \eta
\end{array}}} &
{\begin{array}{c}
A \quad\; A \\
\otimes\otimes\otimes \\
A \quad\; A
\end{array}}
\ar[r]^{\triangledown\otimes\triangledown}
 & {A\otimes A} \ar[r]^{\triangledown'} & A
 }
\end{minipage}
\end{center}
the right parallelogram commutes since $\triangledown'$ is an algebra morphism, the other cells commute trivially. Furthermore both horizontal compound arrows from $A\otimes A$ to $A\otimes A$ are the identity. So $\triangledown'\circ t=\triangledown$ is proved.
The remaining equation $\triangledown'=\triangledown$ proves similarly, replacing $t$ with $\id$ and interchanging $K$ with $A$ in the left capital twist $T$.
\end{proof}

\begin{corollary}
\label{element}
\begin{enumerate}
\item
The multiplication $\triangledown:{\cal A}\otimes {\cal A} \rightarrow {\cal A}$ of an algebra ${\cal A}$ is an algebra morphism if and only if ${\cal A}$ is a commutative algebra.
\item
The comultiplication $\vartriangle :{\cal C} \rightarrow {\cal C}\otimes {\cal C}$ of a coalgebra ${\cal C}$ is a coalgebra morphism if and only if $\,{\cal C}$ is cocommutative.
\end{enumerate}
\end{corollary}

\begin{proof}
1. If $\triangledown$ is an algebra morphism, apply Theorem \ref{EckmannHilton} with $\triangledown'=\triangledown$ to see that ${\cal A}$ is commutative. The converse is well known and proves easily.
2. proves with the dual of Theorem \ref{EckmannHilton}.
\end{proof}

Often comultipication and counit of a coalgebra ${\cal C}=(C,\vartriangle,\epsilon)$ are defined by means of a specific base ${\cal B}$ of the linear space $C$.

\begin{example}
\label{diagonal0}
Let $C$ be a $K$-linear space and ${\cal B}$ a base of $C$. Define linear applications
\begin{eqnarray*}
\vartriangle_{\cal B} :C &\rightarrow& C \otimes C\,,\qquad \vartriangle_{\cal B}(b) := b\otimes b \quad \mbox{for } b\in {\cal B}\,, \\
\epsilon_{\cal B} :C &\rightarrow& K\;, \qquad \epsilon_{\cal B}(b):=1_K \quad \mbox{for } b\in {\cal B}\,.
\end{eqnarray*}
then $(C,\vartriangle_{\cal B},\epsilon_{\cal B})$ is a coalgebra.
\end{example}
A converse of Example \ref{diagonal0} holds.

\begin{proposition}
\label{BtB}
The set ${\cal B}:=\{ b\in C\mid \;\vartriangle(b)=b\otimes b,\, \epsilon(b)=1_K\}$ in a coalgebra $(C,\vartriangle,\epsilon)$ is linearly independent.
\end{proposition}
\begin{proof}
(from \cite{Street} Proposition 7.2) Assuming that ${\cal B}$ is linearly dependent, we have to derive a contradiction. Let $m$ be the smallest natural number such that there are pairwise different elements $b_0,b_1,\dots ,b_m\in {\cal B}$ which are linearly dependent. The minimal property implies $b\neq 0$ for all $b\in{\cal B}$, whence $m\ge 1$. The dependence relation $\sum_{i=0}^m \alpha_i b_i =0$ has all coefficients $\alpha_i \neq 0$ by the minimal property of $m$. We may assume $\alpha_0 =-1$ so that
$$
b_0 = \sum_{i=1}^m \alpha_i b_i\,,
$$
and $b_1,\dots ,b_m$ are linearly independent by the minimal property of $m$. Applying the linear application $\vartriangle$ we get
\begin{eqnarray*}
\sum_{i,j=1}^m \alpha_i \alpha_j \; b_i\otimes b_j \;=\; b_0\otimes b_0 \;=\; \vartriangle(b_0) &=& \sum_{i=1}^m \alpha_i \vartriangle(b_i)
\;=\; \sum_{i=1}^m \alpha_i \; b_i\otimes b_i \,.
\end{eqnarray*}
Observing that $\{b_i \otimes b_j\mid i,j=1,\dots,m\}$ is a linearly independent set of vectors, and comparing the coefficients, we see $\alpha_i \alpha_j = 0$ for $i\neq j$ and $\alpha_i^2=\alpha_i$.
But this can happen only if $m=1$ and $\alpha_1=1$, contradicting $b_0\neq b_1$.
\end{proof}

\begin{example}
\label{DecoposCo}
Let $C$ be a linear space with a base ${\cal B}=\{b_n \mid n\in\N_0 \}$ which is totally ordered by it's enumeration. Define for $n\in \N_0$
\begin{eqnarray*}
\vartriangle_{\cal B}(b_n)&:=& \sum_{k=0}^n b_k\otimes b_{n-k}\,, \\
\epsilon_{\cal B}(b_n) &:=& \left\{
\begin{array}{ll}
1 & \mbox{if } \quad n=0 \\
0 & \mbox{else}
\end{array} \right.,
\end{eqnarray*}
Then $(C,\vartriangle_{\cal B},\epsilon_{\cal B})$ is a coalgebra.
\end{example}

\begin{example}
\label{PolynCoalg}
We need characteristic $0$, so let $K$=$\C$.
The algebra $(\C[X],\cdot,1)=(\C[X],\triangledown,\eta)$ of polynomials has two natural bases with the corresponding natural coalgebra structures from Example \ref{DecoposCo}.

For more conveniant representation of applications we use the natural algebra isomorphism
\begin{eqnarray*}
\C[X]\otimes \C[X] \rightarrow \C[X,Y]\;, \qquad X^m\otimes X^n \mapsto X^mY^n\,
\end{eqnarray*}
and identify both spaces. Then the multiplication $\triangledown$ of the algebra of polynomials amounts to setting $Y=X$,
$$
\triangledown(h)(X)=h(X,X)\qquad \mbox{for } h\in \C[X,Y]\,.
$$
\begin{enumerate}
\item Base ${\cal B}_1=\{X^n \mid n\in\N_0 \}$ generates the coalgebra $(\C[X],\vartriangle_1,\epsilon_1)$ with
        $$
    \vartriangle_1(X^n) =\sum_{k=0}^n X^k Y^{n-k} \;.
    $$
\item Base ${\cal B}_2=\{\frac{X^n}{n!} \mid n\in\N_0 \}$ generates the coalgebra $(\C[X],\vartriangle_2,\epsilon_2)$. Since
    \begin{eqnarray}
    \label{PolynComult2}
    \vartriangle_2(X^n) =\sum_{k=0}^n {n\choose k} X^k Y^{n-k} = (X+Y)^n
    \end{eqnarray}
    this coalgebra is sometimes called {\bf binomial coalgebra}. By linearity of $\vartriangle_2$ we get
    $$
    \vartriangle_2(f)(X,Y) = f(X+Y)\qquad \mbox{for } f\in \C[X]\,.
    $$
\end{enumerate}
Both counits coincide, $\epsilon_1=\epsilon_2$, (see Example \ref{DecoposCo}) and
$$\epsilon_i(f)=f(0)\;, \qquad \mbox{for } f\in \C[X]\,.$$
\end{example}

The above examples being all cocommutative, we finally present a non-cocom\-mu\-tative one.

\begin{example}
\label{MatrixBialg}
Let $M_n$ be the linear space of square $n\times n$ matrices on the field $K$ and $\{B_{i,j}\mid i,j=1\dots n\}$ the standard base. Defining
\begin{eqnarray*}
\vartriangle(B_{i,j})&:=& \sum_{k=1}^n B_{i,k}\otimes B_{k,j}\,, \\
\epsilon(B_{i,j}) &:=& \left\{
\begin{array}{ll}
1 & \mbox{if } \quad i=j \\
0 & \mbox{else}
\end{array} \right.
\end{eqnarray*}
$(M_n,\vartriangle,\epsilon)$ becomes a coalgebra, which is not cocommutative. For an arbitrary matrix $A\in M_n$ the counit $\epsilon$ applies $A$ to the trace of $A$.
\end{example}

\section{Algebra structures on spaces of linear applications}
\label{AlgHom}

In this section we study the natural algebra structure on the space $H$ of linear applications from a coalgebra to an algebra.

\begin{theorem}
\label{convAlg}
Let ${\cal A}=(A,\triangledown,\eta)$ be an algebra and ${\cal C}=(C,\vartriangle,\epsilon)$ a coalgebra on a field $K$ and $H:=\Hom_K(C,A)$ the linear space of $K$-linear applications from $C$ to $A$. Define for $f$, $g\in H$ the product, called {\bf convolution},
\begin{eqnarray*}
f\star g &:=& \triangledown \circ\; (f\otimes g) \circ \vartriangle
\end{eqnarray*}
and the unit $$u:=\eta\circ\epsilon\,.$$ Then ${\cal H}({\cal C},{\cal A}):=(H,\star,u)$ is a $K$-algebra, called {\bf convolution algebra} of $({\cal C},{\cal A})$.

If $\,{\cal A}\,$ is commutative or $\;{\cal C}\;$ is cocommutative, then the convolution algebra ${\cal H}({\cal C},{\cal A})$ is commutative, too.
\end{theorem}
The following commuting diagrams illustrate the convolution and the unit.
\begin{center}
\begin{minipage}{6cm}
\xymatrix@=3em{
{C}
\ar[r]^{f\star g}
\ar[d]^{\vartriangle}
& {A}
& C \ar[r]^{\epsilon}
\ar[dr]_{u}
& K \ar[d]^{\eta}
 \\
{C\otimes C}
\ar[r]^{f\otimes g}
& {A\otimes A} \ar[u]^{\triangledown}
&& A
 }
\end{minipage}
\end{center}

\begin{proof}
We first check unitarity regarding the following diagram.
\begin{center}
\begin{minipage}{6cm}
\xymatrix@=4em{
{C}
\ar[rrr]_{f\star u}
\ar[dr]^{\vartriangle}
\ar[ddd]^{\scriptsize{\id_C}}
&&& {A}
 \\
& {C\otimes C}
\ar[r]^{f\otimes u}
\ar[d]^{\scriptsize{\id_C}\otimes\epsilon} &
{A\otimes A} \ar[ur]^{\triangledown}
 \\
& {C \otimes K}
\ar[r]^{f\otimes\scriptsize{\id_K}}
& {A \otimes K}
\ar[u]^{\scriptsize{\id_A}\otimes\eta} \ar[dr]^{\sim}
 \\
{C}
\ar[rrr]^{f} \ar[ur]^{\sim}
&&&
{A}
\ar[uuu]^{\scriptsize{\id_A}}
 }
\end{minipage}
\end{center}
The definition of convolution implies that the upper cell commutes. The inner square  commutes by the definition of $u$. The left cell commutes, since $\epsilon$ is a counit, the right one since $\eta$ is a unit. Obviously, the lower cell commutes.
So, the outer square commutes as well, i.e. $f\star u=f$. The equation $u\star f=f$ proves similarly.

Next we check associativity.
\begin{center}
\begin{minipage}{6cm}
\xymatrix@=4em{
{C}
\ar[rrr]^{(f\star g)\star h}_{f\star (g\star h)}
\ar[dr]^{\vartriangle}
\ar[ddd]^{\vartriangle}
&&& {A}
 \\
& {C\otimes C}
\ar[r]^{{f\otimes (g\star h)}}
\ar[d]^{\scriptsize{\id_C}\otimes \vartriangle} &
{A\otimes A} \ar[ur]^{\triangledown}
 \\
& {C \otimes C\otimes C}
\ar[r]^{f\otimes(g\otimes h)}_{(f\otimes g)\otimes h}
& {A \otimes A\otimes A}
\ar[u]^{\scriptsize{\id_A}\otimes\triangledown} \ar[dr]^{\triangledown\otimes\scriptsize{\id_A}}
 \\
{C\otimes C}
\ar[rrr]^{(f\star g)\otimes h} \ar[ur]^{\vartriangle\otimes\scriptsize{\id_C}}
&&&
{A\otimes A}
\ar[uuu]^{\triangledown}
 }
\end{minipage}
\end{center}
The left cell commutes by coassociativity of $\vartriangle$, the right one by associativity of $\triangledown$. Associativity of the tensor product implies, that the upper and lower specification of the lower arrow of the central square coincide. Then, by definition of convolution the central square commutes and so does the lower cell. Again the definition of convolution implies that the outer large square commutes and also the upper cell (with the inner specification of the arrows). We showed that the two specifications of the upper arrow coincide, $(f \star g)\star h=f \star (g\star h)$.

For the last assertion on commutativity we know $\triangledown\circ t=\triangledown$, whence
$f\star g =\triangledown\circ(f\otimes g)\circ\vartriangle =
\triangledown\circ t\circ(f\otimes g)\circ\vartriangle =
\triangledown\circ(g\otimes f)\circ\vartriangle = g\star f$.
The cocommutative case runs similarly.
\end{proof}

\begin{example}
\label{CoalgDual}
For a coalgebra ${\cal C}=(C,\vartriangle,\epsilon)$ the dual $\,{\cal C}^*:=(C^*,\star\,,\epsilon)$ is an algebra with multiplication   $\vartriangle^*|_{C^*\otimes C^*}$ being the dual of $\vartriangle$ restricted to $C^*\otimes C^* \hookrightarrow (C\otimes C)^*$ and the counit $\epsilon$ of $\,{\cal C}$ being the unit of $\,{\cal C}^*$.
\end{example}

\begin{example}
\label{PowerSeries}
Applying the preceding example to the coalgebras of polynomials (Example \ref{PolynCoalg}) we get natural algebra isomorphisms
\begin{eqnarray*}
((\C[X],\vartriangle_1,\epsilon_1)^*,\star_1\,,1) &\rightarrow& (\C[[X]],\cdot\,,1) \,,\qquad f \mapsto \sum_{n=0}^\infty f(X^n) X^n\;, \\
((\C[X],\vartriangle_2,\epsilon_2)^*,\star_2\,,1) &\rightarrow& (\C[[X]],\cdot\,,1) \,,\qquad
f \mapsto \sum_{n=0}^\infty \frac{f(X^n)}{n!} X^n \;.
\end{eqnarray*}
Here $(\C[[X]],\cdot\,,1)$ denotes the algebra of formal power series.
\end{example}

\begin{example}
\label{MatrixBialg2}
The dual of the coalgebra $(M_n,\vartriangle,\epsilon)$ in Example \ref{MatrixBialg} is canonically isomorphic to the matrix algebra $(M_n,\cdot,U)$, $U$ denoting the $n\times n$ unit matrix. The isomorphism $(M_n,\cdot,U)\rightarrow (M_n,\vartriangle,\epsilon)^*$ maps $A=(a_{i,j})_{i,j=1,\dots, n}$ to the linear form given by $(B_{i,j}\mapsto a_{i,j})_{i,j=1,\dots, n}$, especially $U\mapsto \epsilon$.
\end{example}

Passing from pairs $({\cal C},{\cal A})$ to the convolution algebra ${\cal H}({\cal C},{\cal A})$ is compatible with morphisms, i.e.\ this is a bifunctor, contravariant in the coalgebras and covariant in the algebras.

\begin{proposition}
\label{ConvMorph}
Let ${\cal A},{\cal A}'$ be algebras, ${\cal C},{\cal C}'$ coalgebras and $\alpha:A\rightarrow A'$ an algebra morphism, $\gamma: C'\rightarrow C$ a coalgebra morphism, then
\begin{eqnarray*}
h: \Hom_K(C,A) \rightarrow \Hom_K(C',A')\;,\qquad f\mapsto \alpha\circ f\circ \gamma
\end{eqnarray*}
is an algebra morphism from the convolution algebra ${\cal H}({\cal C},{\cal A})$ to the convolution algebra ${\cal H}({\cal C}',{\cal A}')$.
\end{proposition}

\begin{proof}
We have to show
\begin{eqnarray*}
h(\varphi\star \psi)= h(\varphi) \star h(\psi)\,,\qquad h(u)=u'\,.
\end{eqnarray*}

First the unit is applied to the unit: We know $\alpha\circ\eta=\eta'$ since $\alpha$ is an algebra morphism, and $\epsilon\circ\gamma=\epsilon'$ since $\gamma$ is a coalgebra morphism and $u=\eta\circ\epsilon$. Then $h(u)=\alpha\circ\eta\circ\epsilon\circ\gamma=\eta'\circ \epsilon'= u'$.

Regard the following diagram.

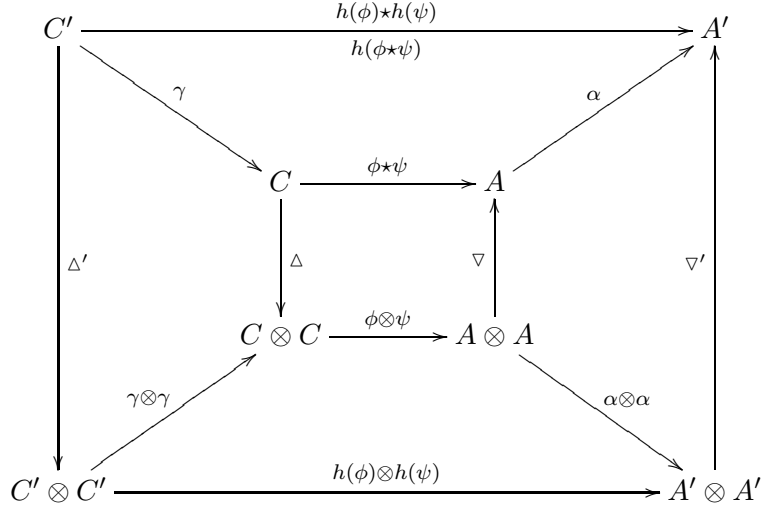
\begin{center}
\begin{minipage}{6cm}
\xymatrix@=4em{
{C'}
\ar[rrr]_{h(\phi\star\psi)}^{h(\phi) \star h(\psi)}
\ar[dr]^{\gamma}
\ar[ddd]^{\vartriangle'}
&&& {A'}
 \\
& C
\ar[r]^{\phi\star\psi}
\ar[d]^{\vartriangle} &
A \ar[ur]^{\alpha}
 \\
& {C \otimes C}
\ar[r]^{\phi\otimes\psi}
& {A \otimes A}
\ar[u]^{\triangledown} \ar[dr]^{\alpha\otimes\alpha}
 \\
{C' \otimes C'}
\ar[rrr]^{h(\phi) \otimes h(\psi)} \ar[ur]^{\gamma\otimes\gamma}
&&&
{A' \otimes A'}
\ar[uuu]^{\triangledown'}
 }
\end{minipage}
\end{center}
The definition of convolution implies that the inner and outer squares commute. The upper and lower cells commute by the definition of $h$ and functoriality of the tensor product. The left cell commutes, since $\gamma$ is a coalgebra morphism, the right since $\alpha$ is an algebra morphism. So, the upper arrow equals $h(\varphi\star \psi)$ (upper cell) and equals also $h(\varphi) \star h(\psi)$ (outer square). We are done.
\end{proof}

The opposite of the convolution algebra is the convolution algebra of the opposites.

\begin{proposition}
\label{opConvol}
Let ${\cal A}$ be an algebra, ${\cal C}$ a coalgebra. Then
$${\cal H}^{op}({\cal C},{\cal A}) = {\cal H}({\cal C}^{op},{\cal A}^{op})\,.$$
\end{proposition}

\begin{proof}
Denote with $\star^{op}$ the convolution $f\star^{op} g = \triangledown^{op} \circ\; (f\otimes g) \circ \vartriangle^{op}$ of the convolution algebra ${\cal H}({\cal C}^{op},{\cal A}^{op})$. We have to show
$$f \star^{op} g = g \star f\,,\qquad f,g\in\Hom_K(C,A)\,.$$
But this is plain since the diagram
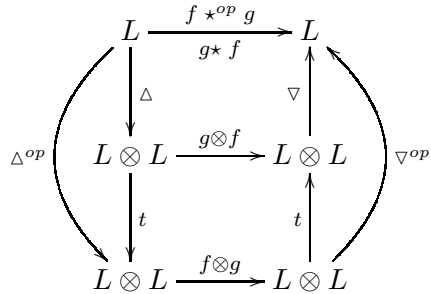
\begin{center}
\begin{minipage}{6cm}
\xymatrix@=3em{
{L}
\ar[r]^{f\;\star^{op}\; g}_{g \star\; f} \ar@/_ 1cm/[dd]_{\vartriangle^{op}}
\ar[d]^{\vartriangle}
& {L}
 \\
{L\otimes L} \ar[d]^t
\ar[r]^{g\otimes f}
& {L\otimes L}  \ar[u]^{\triangledown}
 \\
{L\otimes L}
\ar[r]^{f\otimes g}
& {L\otimes L} \ar[u]^{t} \ar@/_ 1cm/[uu]_{\triangledown^{op}}
 }
\end{minipage}
\end{center}
commutes.
\end{proof}

Convolution is compatible with the tensor product.

\begin{proposition}
\label{TensInv}
Let ${\cal A},{\cal A}'$ be algebras, ${\cal C},{\cal C}'$ coalgebras and let $f,g\in{\cal H}:=(\Hom_K(C,A),\star,u)$ and $f',g'\in {\cal H'}:=(\Hom_K(C',A'),\star,u')$ be applications in the convolution algebras. Then, in the convolution algebra $\underline{\cal H}:=(\Hom_K(C\otimes C',A\otimes A'),\star,u\otimes u')$, one has
\footnote{We use the same symbol $\star$ for the convolution in the three convolution algebras.}
$$
(f\otimes f')\star(g\otimes g')  = (f\star g)\otimes (f'\star g')\,.
$$
If $f$, $g$ are invertible for convolution, then their tensor product is invertible too, and
$$
(f\otimes f')^{\star -1} = f^{\star -1}\otimes f'^{\star -1}\,.
$$
\end{proposition}
\begin{proof}
Regard the following diagram.

\begin{center}
\begin{minipage}{6cm}
\xymatrix@=4em{
{C\otimes C'}
\ar[rrr]_{(f\star g)\otimes(f'\star g')}^{(f\otimes f')\star(g\otimes g')}
\ar[dr]^{\vartriangle\otimes\vartriangle'}
\ar[dd]^{\underline{\vartriangle}}
&&& {A\otimes A'}
 \\
& {\begin{array}{c}
C \quad\; C' \\
\otimes\otimes\otimes  \\
C \quad\; C'
\end{array}}
\ar[r]^{\scriptsize{\begin{array}{c}
f \quad\; f' \\
\otimes\otimes\otimes  \\
g \quad\; g'
\end{array}}}
\ar[dl]^{T} &
{\begin{array}{c}
A \quad\; A' \\
\otimes\otimes\otimes \\
A \quad\; A'
\end{array}}
\ar[ur]^{\triangledown\otimes\triangledown'}
 \\
{\begin{array}{c}
C \otimes C' \\
 \otimes  \\
C \otimes C'
\end{array}}
\ar[rrr]^{\scriptsize{\begin{array}{c}
f \otimes f' \\
 \otimes  \\
g \otimes g'
\end{array}}} &&&
{\begin{array}{c}
A \otimes A' \\
 \otimes  \\
A \otimes A'
\end{array}}
\ar[ul]^{T} \ar[uu]^{\overline{\triangledown}}
 }
\end{minipage}
\end{center}
The defining diagrams for $f\star g$ and $f'\star g'$ together imply that the upper cell with the inner applications commutes. The outer square with $(f\otimes f')\star(g\otimes g')$ on top commutes by definition of convolution in $\underline{\cal H}$. The other three cells commute trivially. So both applications on the top horizontal arrow coincide.

Setting $g=f^{\star-1}$ and $g'=f'^{\star-1}$ the last assertion derives directly from the first.
\end{proof}

\section{Coalgebra structures on spaces of linear applications}
\label{CoalgHom}

Analogous to the last section there is, in general, no natural coalgebra structure on the space of linear applications from an algebra to a coalgebra. The meager remaining result is Example \ref{dualofAlg}.
In this section we rather study coalgebra structures on subspaces of $H= \Hom_K(L,L')$. They require a preassigned base but need only very few structure on $L$, $L'$.
\begin{proposition}
\label{diagonal}
Let $L$ be a $K$-linear space endowed with a unit $\eta: K\rightarrow L$, $1_L:=\eta(1_K)$, and let $L'$ be a space with unit $\eta': K\rightarrow L'$ and counit $\epsilon': L'\rightarrow K$, and suppose $\epsilon'\circ\eta' = \id_K$. Let ${\cal B}\subset H:= \Hom_K(L,L')$ be a linear independent set with $b\circ\eta=\eta'$ for all $b\in{\cal B}$. The span $H_{\cal B}\subseteq H$ of ${\cal B}$ becomes a coalgebra $(H_{\cal B},\vartriangle_{\cal B},\epsilon_{\cal B})$ with
\begin{eqnarray*}
\vartriangle_{\cal B}(f) &:=& \sum_{b\in {\cal B}} x_b\; b\otimes b\,,\qquad f=\sum_{b\in {\cal B}} x_b\;b \in H_{\cal B}\;, \\
\epsilon_{\cal B}(f) &:=& \epsilon'\circ f \circ \eta(1_K) \;=\; \epsilon'\circ f(1_L)\,,\qquad f\in H_{\cal B}\;.
\end{eqnarray*}
\end{proposition}

\begin{proof}
This coalgebra structure on the linear space $H_{\cal B}$ is in fact the coalgebra defined in Example \ref{diagonal0} by $\vartriangle_{\cal B}(b)=b\otimes b$ and $\epsilon_{\cal B}(b)=1_K$ for $b\in{\cal B}$.
The formula for $\vartriangle_{\cal B}$ in the proposition follows by linearity. For $b\in {\cal B}$ our assumptions imply
$\epsilon_{\cal B}(b) = 1_K \;=\; \epsilon'\circ \eta'(1_K) \;=\; \epsilon'\circ b \circ \eta(1_K) \;=\; \epsilon'\circ b(1_L)\,$. These equations extend by linearity to $f\in H_{\cal B}$.
\end{proof}

\noindent
Let $L$, $L'$ be linear spaces with unit $\eta$ respectively $\eta'$. A base of a subspace of the morphism space $\Hom_K(L,L')$ with the property $b\circ \eta = \eta'$ for all base elements $b$, will be called a {\bf unitary base}.

\begin{example}
\label{AlgDual}
Let ${\cal A}=(A,\triangledown,\eta)$ be an algebra structure on $L=A$ and set $L'=K$, then $\Hom_K(A,K)=A^*$ is the dual space of $A$. With the notations of Proposition \ref{diagonal} $1_A=\eta(1_K)$, $\eta'=\id_K$, $\epsilon'=\id_K$. We show that there exists a unitary base ${\cal B}$ of $A^*$, i.e.\ $b(1_A)=1_K,\; b\in{\cal B}$.

It is sufficient to show, that $H_1:=\{f\in A^* \mid f(1_A)=1_K \}$ spans $H$. Let $f\in H$. First, in case $f(1_A)\neq 0$ we get $h:=f/f(1_A) \in H_1$ and $f$ is a multiple of $h$. In the remaining case $f(1_A)= 0$ select a function $g\in H_1$ and set $h:=f+g$. Then $h(1_A)=f(1_A)+g(1_A)=1_K$, whence $h\in H_1$, too. So $f$ equals the difference $h-g$ of functions in $H_1$.

With such base $H_{\cal B}=A^*$ so that by Proposition \ref{diagonal}  $(A^*,\vartriangle_{\cal B},\epsilon_{\cal B})$ is a coalgebra and the counit is the evaluation
$\epsilon_{\cal B}(f)=f(1_A)$ of $f\in A^*$ at the unit of $A$.

In case $n=\dim_KA <\infty$ a unitary base of $A^*$ can be given explicitly.  Let $\{a_1,\dots,a_n\}$ be a base of $A$ with $1_A=\sum_{j=1}^n a_j$ (e.g.\ for the $\R$-algebra $\C$ take $a_1=(1+i)/2$, $a_2=(1-i)/2$). Then the dual base ${\cal B}=\{b_1,\dots,b_n\}$ of $A^*$ has the property $\;b_k(1_A)=b_k(a_k)=1_K,\; k=1,\dots,n\,$, i.e.\ ${\cal B}$ is a unitary base.
\end{example}

In Section \ref{AlgHom} we have seen that $\Hom_K(C,A)$ has a natural algebra structure. Here we ask: Is $\Hom_K(A,C)$ a coalgebra? As we know from Example \ref{AlgDual} this problem has no canonical solution, but it solves canonically w.r.t.\ a base. We start with some preparation.
The following well known lemma proves easily by means of the universal property of tensor products.

\begin{lemma}
\label{Rho}
For $K$-linear spaces $A$, $A'$, $C$ and $C'$ let
$$
H:=\Hom_K(A,C),\; H':=\Hom_K(A',C'),\; \underline{H}:=\Hom_K(A\otimes A',C\otimes C')
$$
Define the application
$$
\varrho : H\otimes H' \rightarrow \underline{H}\,,\qquad f{\otimes}f' \mapsto f\underline{\otimes} \,f'\;,
$$
where $f\underline{\otimes} \,f'$ denotes the tensor product of the linear applications $f$, $f'$, whereas $f{\otimes}f'$ denotes the tensor product of elements of the spaces $H$ and $H'$.
Then $\varrho$ is a linear monomorphism and an isomorphism if the spaces have finite dimensions.

Specializing to the dual spaces, i.e.\ $C$, $C' =K$, we get
\begin{eqnarray*}
\varrho : A^*\otimes A'^* \rightarrow (A\otimes A')^*\,,\qquad f{\otimes} f' \mapsto f\underline{\otimes} \,f'\;, \\
f\underline{\otimes} \,f'(a\otimes a')=f(a)f'(a')\,,\qquad a\in A,\, a'\in A'\,.
\end{eqnarray*}
\end{lemma}

\begin{proof}
By the universal property of the tensor product we get bilinear injections
\begin{eqnarray*}
H\times H' \hookrightarrow \Bilin_K(A\times A',C\otimes C')\hookrightarrow \underline{H}\;
\end{eqnarray*}
and $\varrho$ is the corresponding linear application.
Now, if $\varrho(\sum f_i \otimes f_i')=0$, i.e. $\sum f_i \underline{\otimes} f_i'=0$ then $\sum f_i \times f_i'=0$. It's image $\sum f_i \otimes f_i'$ under the canonical bilinear application $H\times H' \rightarrow H\otimes H'$ equals zero, too.

In the finite dimensional case one easily proves $\dim_K H\otimes H' = \dim_K \underline{H}<\infty$, whence $\varrho$ is an isomorphism.

In case of the dual spaces we get $f\underline{\otimes} \,f'(a\otimes a')=f(a)\otimes f'(a')=f(a)f'(a')$ by means of the canonical isomorphism $K\otimes K\simeq K$.
\end{proof}

\begin{example}
\label{dualofAlg}
Let ${\cal A}=(A,\triangledown,\eta)$ be a finite dimensional $K$-algebra.
With the natural isomorphism $\varrho:A^*\otimes A^*\rightarrow (A\otimes A)^*$ from Lemma \ref{Rho} and the dual $\triangledown^*:A^*\rightarrow (A\otimes A)^*$ of the multiplication of ${\cal A}$ we build the application
$$
\vartriangle:=\varrho^{-1}\circ\triangledown^* : A^*\rightarrow A^*\otimes A^* \,.
$$
Then $(A^*,\vartriangle,\varepsilon)$ is a coalgebra, where
$$
\varepsilon(f):=f(1_A)\,,\qquad f\in A^*\,.
$$
\end{example}
How to get explicit formulas for the comultiplication, is shown with further specializing the last example.

\begin{example}
\label{ComplexStar}
We take the
$\R$-algebra $(\C,\triangledown,\eta) = (\C,\,\cdot\,,1)$
of complex numbers. On the dual space
$\,\C^* = \Hom_{\R}(\C,\R)\,$ let $\{h_1,h_2\}$ be the dual base of the standard base $\{1,i\}$ of $\,\C$.  Then
$$
\varepsilon(h_1)=1\,,\qquad \varepsilon(h_2)=0\,.
$$
We write complex numbers as $a=a_1+ia_2$, $a_1,a_2\in\R$. Then we get using $\,h_1\underline{\otimes}h_2(a\otimes b)=a_1\otimes b_2=a_1b_2\,$ etc.\
\begin{eqnarray*}
\triangledown^*(h_1)(a\otimes b) =h_1(ab)=a_1b_1-a_2b_2=(h_1\underline{\otimes}h_1 - h_2\underline{\otimes}h_2)(a\otimes b)\,, \\
\triangledown^*(h_2)(a\otimes b) =h_2(ab)=a_1b_2+a_2b_1=(h_1\underline{\otimes}h_2 + h_2\underline{\otimes}h_1)(a\otimes b)
\end{eqnarray*}
and, applying $\varrho^{-1}$,
\begin{eqnarray*}
\vartriangle(h_1)=\varrho^{-1}\circ\triangledown^*(h_1)=h_1{\otimes}h_1 - h_2{\otimes}h_2\,, \\
\vartriangle(h_2)=\varrho^{-1}\circ\triangledown^*(h_2)=h_1{\otimes}h_2 + h_2{\otimes}h_1\,.
\end{eqnarray*}
\end{example}
With the $\R$-linear isomorphism $h_1 \mapsto 1$, $h_2\mapsto i$ between $\C^*$ and $\C$ we get

\begin{example}
\label{Complex}
The
$\R$-linear space $\C$ becomes a coalgebra $(\C,\vartriangle,\varepsilon)$ setting
\begin{eqnarray*}
\vartriangle(1)=1\otimes 1 - i\otimes i\,, &\qquad&
\varepsilon(1)=1\,, \\
\vartriangle(i)=1\otimes i + i\otimes 1\,,
&\qquad& \varepsilon(i)=0\,.
\end{eqnarray*}
\end{example}

\begin{proposition}
\label{split}
Let ${\cal A}=(A,\triangledown,\eta)$ and  ${\cal A}'=(A',\triangledown',\eta')$ be algebras and ${\cal C}=(C,\vartriangle,\epsilon)$ and ${\cal C}'=(C',\vartriangle',\epsilon')$ coalgebras. With the notations from Lemma \ref{Rho} define the {\bf (non-linear) split mapping}
$$
s: \underline{H} \rightarrow H\otimes H' \,, \qquad F \mapsto F_1 \otimes F_2\,,
$$
where the linear applications $F_1: A\rightarrow C$, $F_2:A'\rightarrow C'$ are defined by
\begin{eqnarray*}
F_1(x)&:=& (\id_C \otimes \epsilon')\circ F(x\otimes 1_{A'})\,,\quad \mbox{for } x\in A \,, \\
F_2(x')&:=& (\epsilon \otimes \id_{C'})\circ F(1_A\otimes x')\,,\quad \mbox{for } x'\in A'\,.
\end{eqnarray*}
Then $s$ is quadratic, $s(k F)=k^2 s(F)$, $k\in K$, and
\begin{eqnarray}
\label{scircrho}
s(f\underline{\otimes} f') = \epsilon\circ f(1_A)\cdot \epsilon'\circ f'(1_{A'}) \;\;f{\otimes}f'\,, \qquad f\in H, f'\in H'\;.
\end{eqnarray}
\end{proposition}

\begin{proof}
First $s(k F)=(kF_1)\otimes (kF_2) = k^2 F_1\otimes F_2 = k^2 s(F)$.

We set $F:= \varrho(f\otimes f')=f\underline{\otimes} f'$. For $x\in A$, $x'\in A'$ we then have  $F(x\otimes x')=f(x)\otimes f'(x')$. Since $\varrho$ is injective it is sufficient to prove, that both sides of equation (\ref{scircrho}) have equal images under $\varrho$.
\begin{eqnarray*}
\varrho\circ s(f\underline{\otimes} f')(x\otimes x') &=& F_1\,\underline{\otimes}\, F_2\;(x\otimes x') \\
&=& (\id_C \otimes \epsilon')\circ F(x\otimes 1_{A'}) \otimes (\epsilon \otimes \id_{C'})\circ F(1_A\otimes x') \\
&=& (f(x)\otimes \epsilon'(f'(1_{A'})) \otimes (\epsilon(f(1_{A})\otimes f'(x')) \\
&=& \epsilon'(f'(1_{A'})) \cdot \epsilon(f(1_{A}))\;\; f\underline{\otimes} f'\; (x\otimes x')  \\
&=& \varrho\circ (\epsilon'(f'(1_{A'})) \cdot \epsilon(f(1_{A}))\;\; f\otimes f')\, (x\otimes x')  \,.
\end{eqnarray*}
For the fourth equality we used the canonical isomorphisms $C\otimes K\simeq C$, $K\otimes C\simeq C$.
Since these equations hold for all $x\otimes x' \in A\otimes A'$, we are done.\footnote{Despite $s$ not being linear, be aware that $\varrho\circ s(F)$ ia a linear application.}
\end{proof}
\noindent
In (\ref{scircrho}) the constant factor on the right hand side can be zero. We are looking for conditions, such that this factor becomes $1$.

\begin{corollary}
\label{affineBase}
Under the assumptions of Proposition \ref{split} suppose that the underlying spaces of ${\cal A}$ and ${\cal C}$ are identical, $A=C $, and that counit and unit are compatible,
\begin{eqnarray*}
\epsilon\circ\eta = \id_K\,. \qquad \epsilon'\circ\eta' = \id_K\,.
\end{eqnarray*}
Regard the affine spaces
\begin{eqnarray*}
H_1&:=&\{f\in H\mid f(1_A)=1_A \}\,,\qquad H'_1\quad :=\quad \{f'\in H'\mid f'(1_{A'})=1_{A'} \}\,
\\ \underline{H}_1&:=&\{F\in \underline{H}\mid F(1_{A\otimes A'})=1_{A\otimes A'}\}
\end{eqnarray*}
of linear mappings respecting the unit. Then
$$
s({\underline{H}_1}) \subseteq H_1\otimes H_1'\;
$$
$$
s\circ \varrho\,(f\otimes f')  = s\,(f\underline{\otimes}f') = f\otimes f'\;,\qquad f\in H_1, f'\in H_1'\;.
$$
\end{corollary}
\noindent

\begin{proof}
If $F\in\underline{H}_1$ then, using compatibility of counit and unit, $F_2(1_{A'})=(\epsilon \otimes \id_{C'})\circ F(1_A\otimes 1_{A'})=(\epsilon \otimes \id_{C'})(1_A\otimes 1_{A'})=1_{A'}$  that is $F_2\in H_2$ and similarly $F_1\in H_1$. We get $s(F)=F_1\otimes F_2 \in H_1\otimes H_1'\,$.
The last assertion is just (\ref{scircrho}), where the coefficients on the right hand side equal $1$ by our assumptions.
\end{proof}

Our question "Is $\Hom_K(A,C)$ a coalgebra?" can now be answered affirmatively, but only on suitable subspaces.

\begin{proposition}
\label{DeltaCirc}
Let $L$ ba a $K$-linear space, which supports an algebra ${\cal A}=(L,\triangledown,\eta)$ and a coalgebra $\,{\cal C}=(L,\vartriangle,\epsilon)$.
Suppose that the structures of ${\cal A}$ and ${\cal C}$ are compatible in the weak sense, that the unit $\eta$ is a coalgebra morphism, i.e.\
\begin{eqnarray}
\label{etaCoMorph}
\epsilon\circ\eta = \id_K\,, \qquad \vartriangle \circ\,\eta = \eta \otimes \eta\,.
\end{eqnarray}
Like above set
\begin{eqnarray*}
H :=& \Hom_K(L,L),\; \qquad\qquad & \underline{H}:= \Hom_K(L\otimes L,L\otimes L) \\
H_1 :=& \{f\in H\mid f(1_A)=1_A \}\,,\qquad &
\underline{H}_1 := \{F\in \underline{H}\mid F(1_{{\cal A}\otimes {\cal A}})=1_{{\cal A}\otimes {\cal A}} \}
\end{eqnarray*}
Then $\vartriangle\circ f\circ \triangledown  \in \underline{H}_1$ if $f\in H_1$ and (recall that $s$ is not linear)
\begin{eqnarray*}
s(\vartriangle\circ f\circ \triangledown) = f\otimes f \qquad \mbox{for } f\in H_1\,.
\end{eqnarray*}
Restricting this application $f\mapsto f\otimes f$ to a linearly independent set  $\,{\cal B} \subseteq H_1$ generates a comultiplication on the span $H_{\cal B}\subseteq H$ of $\,{\cal B}$. Namely defining $\vartriangle_{\cal B}$ and $\epsilon_{\cal B}$ like in Proposition \ref{diagonal} then $(H_{\cal B},\vartriangle_{\cal B},\epsilon_{\cal B})$ is a coalgebra. \end{proposition}

\begin{proof}
Let $f\in H_1$ then (\ref{etaCoMorph}) immediately implies $\vartriangle\circ f\circ \triangledown  \in \underline{H}_1$. That $(H_{\cal B},\vartriangle_{\cal B},\epsilon_{\cal B})$ is a coalgebra, follows from Proposition \ref{diagonal}.
\end{proof}

\section{Bialgebras and Hopf algebras}
\label{BiHopf}

Here we collect the definitions and basic properties of bilagebras and Hopf algebras, weakening the compatibility condition between multiplication and comultiplication in order to access a greater domain of applications. Some of them, namely m-weak incidence Hopf algebras, will be treated in the remaining sections.

Let $L$ ba a $K$-linear space, which supports an algebra ${\cal A}=(L,\triangledown,\eta)$ and a coalgebra ${\cal C}=(L,\vartriangle,\epsilon)$. We regard the following compatibility conditions.
\begin{eqnarray}
\label{bi1}
(\triangledown\otimes\triangledown)\circ T\circ(\vartriangle\otimes\vartriangle) &=& \vartriangle\circ\triangledown \quad : L\otimes L \rightarrow L\otimes L \\
\label{bi2}
\vartriangle \circ\,\eta &=& \eta \otimes \eta \quad : K \simeq K\otimes K \rightarrow L\otimes L \\
\label{bi3}
\epsilon\circ \triangledown &=& \epsilon\otimes\epsilon \quad : L\otimes L \rightarrow K \simeq K\otimes K \\
\label{bi4}
\epsilon\circ\eta &=& \id_K \quad : K\rightarrow K
\end{eqnarray}
Recall that $\overline{\triangledown}=(\triangledown\otimes\triangledown)\circ T$ is the multiplication of the algebra ${\cal A}\otimes {\cal A}$ and $\underline{\vartriangle}=T\circ(\vartriangle\otimes\vartriangle)$ the comultiplication of the coalgebra $\;{\cal C}\otimes {\cal C}$. Condition (\ref{bi1}) is represented by the commuting diagram
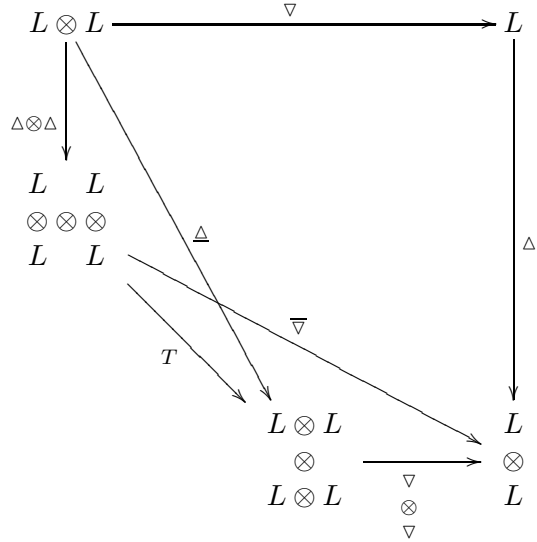
\begin{center}
\begin{minipage}{6cm}
\xymatrix@=4em{
{L\otimes L}
\ar[rr]^{\triangledown}
\ar[d]_{\vartriangle\otimes\vartriangle}
\ar[ddr]^{\underline{\vartriangle}}
&& {L}
\ar[dd]^{\vartriangle}
 \\
{\begin{array}{c}
L \quad\; L \\
\otimes\otimes\otimes  \\
L \quad\; L
\end{array}}
\ar[dr]_{T} \ar[drr]^{\overline{\triangledown}}
 \\
& {\begin{array}{c}
L \otimes L \\
\otimes \\
L \otimes L
\end{array}}
\ar[r]_{\scriptsize{\begin{array}{c}
\triangledown \\
\otimes \\
\triangledown
\end{array}}} &
{\begin{array}{c}
L \\
\otimes \\
L
\end{array}}
 }
\end{minipage}
\end{center}
and it is obvious that
\begin{eqnarray*}
\vartriangle \mbox{ is an algebra morphism } &\Leftrightarrow & (\ref{bi1}), (\ref{bi2}) \mbox{ hold} \\
\epsilon \mbox{ is an algebra morphism } &\Leftrightarrow & (\ref{bi3}), (\ref{bi4}) \mbox{ hold} \\
\triangledown \mbox{ is a coalgebra morphism } &\Leftrightarrow & (\ref{bi1}), (\ref{bi3}) \mbox{ hold} \\
\eta \mbox{ is a coalgebra morphism } &\Leftrightarrow & (\ref{bi2}), (\ref{bi4}) \mbox{ hold}
\end{eqnarray*}
The quintuple $(L,\triangledown,\eta;\vartriangle,\epsilon)$ is called a {\bf m-weak bialgebra} if $(L,\triangledown,\eta)$ is an algebra and $(L,\vartriangle,\epsilon)$ a coalgebra and (\ref{bi2}), (\ref{bi3}), (\ref{bi4}) hold. A m-weak bialgebra is called a (strong) {\bf bialgebra} if equation (\ref{bi1}) holds, too.

\begin{example}
\label{MatrixBialg3}
The coalgebra $(M_n,\vartriangle,\epsilon)$ of $n\times n$ matrices (Example \ref{MatrixBialg}) and it's dual matrix algebra $(M_n,\,\cdot\,,U)$ (Example \ref{MatrixBialg2}) forms a bialgebra. The proof is straightforward using the property $\;B_{i,j}\cdot B_{k,\ell}=\delta_{j,k}B_{i,\ell}\;$ of the standardbase $\{B_{i,j} \mid i,j=1,\dots ,n\}$ of $M_n$.
\end{example}

Be aware, that in a bialgebra the multiplication $\triangledown$ is a coalgebra morphism, but it is an algebra morphism only in the commutative case (Corollary \ref{element}), similarly with the comultiplication. Does there exist a more symmetric issue between m-weak and strong bialgebra?

A {\bf bialgebra morphism} is an application between m-weak bialgebras which is simultaneously an algebra morphism and a coalgebra morphism.

The linear space $H= \Hom_K(L,L)$ supports two algebra structures. First the convolution algebra ${\cal H}({\cal C},{\cal A}):=(H,\star,u)$ (cf.\ Theorem \ref{convAlg}). The other one is simpler with composition of applications as multiplication and the identity as unit, $(H,\circ,\id_L)$. We call it the {\bf composition algebra} of the linear space $L$.

\begin{proposition}
\label{composBialg}
Let $L$ ba a $K$-linear space endowed with linear applications $\eta: K\rightarrow L$ and $\epsilon: L\rightarrow K$, which are compatible, i.e.\ (\ref{bi4}) holds. Let $H_{\cal B}$ be a subalgebra of the composition algebra $(\Hom_K(L,L),\circ,id_L)$, which has a unitary base ${\cal B}$, i.e.\
$b\circ\eta=\eta$ for all $b\in {\cal B}$. Then $H_{\cal B}$ becomes a bialgebra,
$$
{\cal H}_{\cal B} :=(H_{\cal B},\circ,\id_L;\vartriangle_{\cal B},\epsilon_{\cal B})\,,
$$
where $(H_{\cal B},\vartriangle_{\cal B},\epsilon_{\cal B})$ is the coalgebra defined in Proposition \ref{diagonal},
$$\vartriangle_{\cal B}(b) := b\otimes b\,, \quad
\epsilon_{\cal B}(b):=1_K\,, \qquad b\in {\cal B}\,.$$
\end{proposition}
Examples for such bases and subspaces will be given in the context of incidence algebras, see Proposition \ref{compBialg}. \\

\begin{proof}
We denote multiplication and unit of ${\cal H}_{\cal B}$ with $\triangledown_\circ$ and $\eta_\circ$, respectively. It is sufficient to verify (\ref{bi1}) through (\ref{bi4}) for base vectors.  $\{ b\otimes b' \mid b, b' \in {\cal B} \}$ is a base of $H_{\cal B}\otimes H_{\cal B}$.
\begin{eqnarray*}
(\triangledown_\circ\otimes\triangledown_\circ)\circ T\circ(\vartriangle_{\cal B}\otimes\vartriangle_{\cal B})\;(b\otimes b') &=& (b\circ b')\otimes (b\circ b') \; =\;  \vartriangle_{\cal B}\circ\triangledown_\circ \,(b\otimes b') \\
\vartriangle_{\cal B} \circ\,\eta_\circ (1_K) &=& \id_L\otimes\id_L \;=\; \eta_\circ \otimes \eta_\circ\, (1_K\otimes 1_K) \\
\epsilon_{\cal B}\circ \triangledown_\circ\;(b\otimes b') &=& 1_K\cdot 1_K \;=\; \epsilon_{\cal B}\otimes\epsilon_{\cal B}\;(b\otimes b') \\
\epsilon_{\cal B}\circ\eta_\circ (1_K) &=& 1_K
\end{eqnarray*}
The proof is complete.
\end{proof}

In analogy to Proposition \ref{composBialg} we ask: Does there exist a natural coalgebra structure on $H= \Hom_K(L,L)$, such that, together with the convolution algebra $(H,\star,u)$, we get a bialgebra? The examples in Section \ref{Examples} will give an affirmative answer, but like in Proposition \ref{composBialg} only on a subspace of $H$ and the coalgebra structure depending on a preassigned base.

By definition a {\bf (m-weak) Hopf algebra} is a (m-weak) bialgebra $(L,\triangledown,\eta;\vartriangle,\epsilon)$ for which the unit $\id_L$ of the composition algebra $(H,\circ,\id_L)$ has a multiplicative inverse in the convolution algebra ${\cal H}=(H,\star,u)$.\footnote{The analogous question, if $u$ has an inverse in the composition algebra, can be affirmative only in case $\dim_K L=1$ since $u=\eta\circ\varepsilon$.}
This inverse $S:=\id_L^{\star-1}$ is called the {\bf antipode} of the (m-weak) Hopf algebra. Recall that the antipode $S$, if it exists, is uniquely determined by the two equations
\begin{eqnarray}
\label{Antipode}
\id_L \star S = u \,,\quad S\star \id_L = u \,.
\end{eqnarray}

The bialgebra ${\cal H}_{\cal B}$ in Proposition \ref{composBialg} has no antipode, whence it is no Hopf algebra. The same holds for the matrix bialgebra in Example \ref{MatrixBialg3}. But restricting to upper triangular matrices, one gets m-weak Hopf algebras as we shall see in the next sections.

\begin{example}
\label{Polynbialg}
With the notations of Example \ref{PolynCoalg} $(\C[X],\cdot,1;\vartriangle_2,\epsilon)$ is a bialgebra, whereas $(\C[X],\cdot,1;\vartriangle_1,\epsilon)$ is only a m-weak bialgebra since equation (\ref{bi1}) fails.

$\vartriangle_1$ is no algebra morphism:
$
\vartriangle_1(X^m X^n)=\sum_{k=0}^{m+n} X^k\otimes X^{m+n-k}
$
whereas
$
\vartriangle_1(X^m)\vartriangle_1( X^n)=\sum_{i=0}^m \sum_{j=0}^n X^{i+j}\otimes X^{m+n-(i+j)} = \sum_{k=0}^{m+n} a_k X^k\otimes X^{m+n-k}\,$, with $a_k:=|\{(i,j)\in\N_0\times\N_0\mid i\le m, j\le n, i+j=k\}|\,$.

Both m-weak bialgebras have an antipode, namely
\begin{eqnarray*}
S_1(X^n) = \left\{
\begin{array}{ll}
1 & \mbox{if } \quad n=0 \\
-X & \mbox{if } \quad n=1 \\
0 & \mbox{else}
\end{array} \right.
\;,\qquad S_2(X^n)=(-X)^n =S_2(X)^n\;.
\end{eqnarray*}
$S_2$ is a bialgebra morphism. $S_1$ is neither an algebra nor a coalgebra morphism.
\end{example}

\begin{example}
\label{Hamilton}
Hamilton's skew field $\,\Quat$ of quaternions can be made a m-weak Hopf algebra, if one takes the coalgebra structure from Example \ref{diagonal0} w.r.t.\ the natural base ${\cal B}=\{1,i,j,k\}$ of $\,\Quat$ over the field $\,\R$ of reals. It is no Hopf algebra since (\ref{bi1}) fails, $(\triangledown\otimes\triangledown)\circ T\circ(\vartriangle\otimes\vartriangle)(i\otimes i)= (-1)\otimes (-1) \neq -1\otimes 1 = \vartriangle\circ\triangledown(i\otimes i)$. The antipode $S$ is conjugation, $\;S(a+ib+jc+kd)=a-ib-jc-kd,\quad a,b,c,d\in\R$. One easily checks that $S(x\cdot y)=S(y)\cdot S(x)$, i.e.\ $S:\Quat\rightarrow \Quat\;$ is an antimorphism of $\,\R$-algebras.

The sub-$\R$-algebra $\C\subset \Quat$ of complex numbers is a commutative m-weak Hopf algebra.
\end{example}

The antipode is compatible with many operations and constructions as we will show next. First, the opposite bialgebra has the same antipode as the original one.

\begin{proposition}
\label{opAntipode}
The opposite $(L,\triangledown^{op},\eta;\vartriangle^{op},\epsilon)$ of a (m-weak) bialgebra $(L,\triangledown,\eta; \\
\vartriangle,\epsilon)$ is a (m-weak) bialgebra. The opposite of a m-weak Hopf algebra with antipode $S$ is again a m-weak Hopf algebra with the same antipode $S$.
\end{proposition}

\begin{proof}
Since the field $K$ is commutative, the first assertion is trivial for m-weak bialgebras. To prove condition (\ref{bi1}) for $(L,\triangledown^{op},\eta;\vartriangle^{op},\epsilon)$ we get from the diagram (\ref{bi1}) for $(L,\triangledown,\eta;\vartriangle,\epsilon)$ and from $\overline{\triangledown}\circ t = \overline{\triangledown\circ t }$ (see (\ref{querop})), that $\vartriangle$ is an algebra morphism w.r.t.\ $\triangledown^{op}$. That is, $(L,\triangledown^{op},\eta;\vartriangle,\epsilon)$ is a bialgebra. Dually, $(L,\triangledown,\eta;\vartriangle^{op},\epsilon)$ is a bialgebra, too. Starting again with the last bialgebra we see that $\vartriangle^{op}$ is an algebra morphism w.r.t.\ $\triangledown^{op}$, i.e.\ $(L,\triangledown^{op},\eta;\vartriangle^{op},\epsilon)$ is a bialgebra.

Using Proposition \ref{opConvol} and (\ref{Antipode}) we get
$$\id \star^{op} S = S \star \id =u\,,\quad S \star^{op} \id = \id \star S =u$$
which shows that $S$ is the $\star^{op}$-inverse of $\id$, too.
\end{proof}

The tensor product preserves bialgebras and the antipode.

\begin{proposition}
\label{TensAntip}
If $B$, $B'$ are (m-weak) bialgebras then their tensor product is a (m-weak) bialgebra. If $B$, $B'$ are (m-weak) Hopf algebras with antipodes $S$ and $S'$, respectively, then $B\otimes B'$ is a (m-weak) Hopf algebra with antipode $S\otimes S'$.

Also, the tensor product of bialgebra morphisms is a bialgebra morphism.
\end{proposition}
\begin{proof}
The tensor product of algebras is an algebra and the same holds with coalgebras (see Section 2). Since the counits $\epsilon$, $\epsilon'$ are algebra morphisms, so is $\epsilon\otimes\epsilon'$ by Proposition \ref{TensMorph}. Similarly the unit $\eta\otimes\eta'$ is a coalgebra morphism. That is, conditions (\ref{bi2}), (\ref{bi3}), (\ref{bi4}) are inherited by the tensor product. So the tensor product of m-weak bialgebras is a m-weak bialgebra. In case of strong bialgebras it remains to prove that (\ref{bi1}) is inherited, too. Again Proposition \ref{TensMorph} implies that the comultiplication $\underline{\vartriangle}=T\circ(\vartriangle\otimes\vartriangle')$ of $B\otimes B'$ is an algebra morphism.

Furthermore, if antipodes exist, Proposition \ref{TensInv} implies
$
\;\id_{B\otimes B'}^{\star -1} = (\id_B\otimes \id_{B'})^{\star -1} = \id_B^{\star -1}\otimes \id_{B'}^{\star -1} = S\otimes S'\,.
$

Finally, Proposition \ref{TensMorph} implies the assertion on morphisms.
\end{proof}

The operations composition and convolution obey certain distributivity laws. Especially the convolution inverse of (co)algebra morphisms can be computed by composition with the antipode.

\begin{proposition}
\label{InvAntip}
\begin{enumerate}
\item
Let $\alpha\in \Hom_K(B,A)$ be an algebra morphism from a m-weak bialgebra $B$ to an algebra $A$, then
\begin{eqnarray*}
\alpha\circ (f\star g) &=& (\alpha\circ f)\star(\alpha\circ g)\,,\quad f,g\in (\Hom_K(B,B),\star,u)\,.
\end{eqnarray*}
Especially $\alpha\circ S $ is the convolution inverse of $\alpha$, if $B$ is a m-weak Hopf algebra with antipode $S$.
\item
Let $\gamma\in \Hom_K(C,B)$ be a coalgebra morphism from a coalgebra $C$ to a m-weak bialgebra $B$, then
\begin{eqnarray*}
(f\star g)\circ \gamma &=& (f\circ \gamma)\star(g\circ \gamma)\,,\quad f,g\in (\Hom_K(B,B),\star,u)\,.
\end{eqnarray*}
Especially $S \circ \gamma$ is the convolution inverse of $\gamma$, if $B$ is a m-weak Hopf algebra with antipode $S$.
\end{enumerate}
\end{proposition}

\begin{proof}
1. Applying Proposition \ref{ConvMorph} with $\gamma=\id_B$ we get an algebra morphism
\begin{eqnarray*}
h: \Hom_K(B,B) \rightarrow \Hom_K(B,A)\;,\qquad f\mapsto \alpha\circ f \,,
\end{eqnarray*}
whence
$
\alpha\circ (f\star g) = h(f\star g)=h(f)\star h(g)= (\alpha\circ f)\star(\alpha\circ g)$.
Denoting with $u_A$ the unit of the right hand side convolution algebra, we get with (\ref{Antipode})
$
u_A = h(u)=h(\id_B\star S) =\alpha \star (\alpha\circ S)\,
$
and $u_A= h(u)= h(S\star \id_B) = (\alpha\circ S)\star \alpha\,$.

2. proves similarly.
\end{proof}

Bialgebra morphisms respect the antipode.
\begin{corollary}
Let $\beta :B\rightarrow B'$ be a bialgebra morphism between m-weak Hopf algebras with antipodes $S$ and $S'$, then
$$
\beta \circ S \;=\; S'\circ \beta\,.
$$
\end{corollary}

\begin{proof}
Proposition \ref{InvAntip} implies $\beta \circ S = \beta^{\star-1} = S'\circ \beta$.
\end{proof}

According to their importance in Proposition \ref{InvAntip} we regard the algebra morphisms among the linear morphisms in the morphism space $H$. Applying the unit to the unit enforces that they form an affine subspace, but no linear subspace (cf. Corollary \ref{affineBase}). We show that in the abelian case this affine space forms a group with convolution as group operation (cf.\ \cite{Manin} 2.5).

\begin{proposition}
\label{starAlgMorph}
Let $(B,\triangledown,\eta;\vartriangle,\epsilon)$ be a bialgebra, $(A',\triangledown',\eta')$ an abelian algebra and
$$
H_{\mbox{\small alg}}:= \Hom_{K\mbox{\small -alg}}(B,A') \;\subseteq\; H=\Hom_K(B,A')
$$
the affine subspace of the $K$-algebra morphisms $\varphi: B\rightarrow A'$. Then
$$ \varphi,\psi \in H_{\mbox{\small alg}} \;\Rightarrow\; \varphi\star\psi \in H_{\mbox{\small alg}}\, $$
and $(H_{\mbox{\small alg}},\star,u)$ is an abelian monoid.

$(H_{\mbox{\small alg}},\star,u)$ is an abelian group if $(B,\triangledown,\eta;\vartriangle,\epsilon)$ is a Hopf algebra.
\end{proposition}

\begin{proof}
$\varphi\otimes\psi$ is an algebra morphism by Proposition \ref{TensMorph},
$\vartriangle$ is an algebra morphism by the bialgebra property (\ref{bi1}) and $\triangledown'$ is one since the algebra $A'$ is abelian (Proposition \ref{element}), whence $\varphi\star\psi = \triangledown' \circ (\varphi\otimes\psi)\,\circ\vartriangle$ is an algebra morphism as well. Since $A'$ is abelian, the monoid $(H_{\mbox{\small alg}},\star,u)$ is abelian, too (Theorem \ref{convAlg}).

Under the final assumption, $B$ has an antipode and Proposition \ref{InvAntip} implies that any $\varphi\in (H_{\mbox{\small alg}},\star,u)$ has an inverse.
\end{proof}

\begin{corollary}
\label{ScircS}
Let $B$ be an abelian Hopf algebra with antipode $S$, then $S$ is an algebra isomorphism and $$\;S\circ S \,=\, \id_B\,.$$
\end{corollary}
Example \ref{Polynbialg} shows that the corollary fails for m-weak Hopf algebras. \\

\begin{proof}
Clearly $\id_B \in H_{\mbox{\small alg}}$, so Proposition \ref{starAlgMorph} implies $S=\id^{\star -1} \in H_{\mbox{\small alg}}$. Applying now Proposition \ref{InvAntip} with $\varphi=S$ we get $S\circ S=S^{\star -1}=\id_B$.
\end{proof}

\noindent
In the non abelian case one only gets (cf. Example \ref{Hamilton})
\begin{theorem}
\label{SAnti}
Let $B$ be a Hopf algebra with antipode $S$, then $S$ is an antimorphism of algebras and coalgebras, i.e.\
\begin{eqnarray*}
S\circ\triangledown = \triangledown^{op}\circ (S\otimes S)\,, \qquad \vartriangle^{op}\circ S = (S\otimes S)\circ\vartriangle\,.
\end{eqnarray*}
\end{theorem}

\begin{proof}
\cite{Street} Proposition 9.1 (b).
\end{proof}

Finally an example to illustrate Proposition \ref{starAlgMorph}.

\begin{example}
\label{Polynbialg2}
We resume the Hopf algebra $(\C[X],\cdot,1;\vartriangle_2,\epsilon)$ of Example \ref{Polynbialg}. Any polynomial $g\in\C[X]$ induces via $X\mapsto g$ an algebra endomorphism $$\varphi_g:\C[X]\rightarrow\C[X]\,,\quad f\mapsto f\circ g$$
and any algebra endomorphism of $\C[X]$ is of this form. Using (\ref{PolynComult2}), one easily checks $\varphi_g\star\varphi_h(f)=f\circ(g+h)\,,$ whence
$$
\varphi_g\star\varphi_h = \varphi_{g+h}\,, \qquad \varphi_0=u \,.
$$
So we have a group isomorphism
$$(\C[X],+,0)\rightarrow (H_{\mbox{\small \C -alg}}(\C[X],\C[X]),\star,u)\,,\quad g\mapsto \varphi_g\,,$$
especially $\varphi_X=\id_{\C[X]}$, $\varphi_g^{\star-1}=\varphi_{-g}$, $S_2=\varphi_{-X}$.
Notice that $\varphi_g$ is a coalgebra morphism only if $g(X)=aX$ with $a\in \C$. Hence $S_2$ is one as we know already from Example \ref{Polynbialg}.
\end{example}

\section{Posets and their interval spaces}
\label{Poset}

The linear space, spanned by the intervals of a locally finite ordered set, supports canonically a multiplication and a comultiplication. Under a unitarian condition it becomes a m-weak bialgebra. This holds, too, if one passes to the quotient space w.r.t.\ a suitable equivalence relation on intervals. Many important examples are addressed.

Let ${\cal P}=(P,\preceq)$ be a countable or finite partially
ordered set (poset for short). We always suppose that it is {\bf locally
finite}, i.e.\ all {\bf intervals}
$$[a,b]:=\{x\in P \mid a\preceq x\preceq b\},\quad a,b \in P\,,$$
are finite. We write $[a,b]_\preceq$ for an interval, if it is necessary to indicate the ordering. $\intv({\cal P})$ denotes the set of all intervals $\neq \emptyset$.

Concatenation and decomposition of intervals lead to the following natural definitions.
We denote with $L=L({\cal P})$ the $K$-linear space with base $\intv({\cal P})$.
A {\bf multiplication} $\triangledown :L\otimes L\rightarrow L$ is defined
by the following values on the base
\begin{eqnarray*}
\triangledown([a,b]\otimes [c,d]) :=\left\{
\begin{array}{ll}
[a,d] & \mbox{if} \quad  b=c \\
{\bf 0} & \mbox{else}
\end{array} \right. .
\end{eqnarray*}
Similarly, a {\bf comultiplication} $\vartriangle :L\rightarrow L\otimes L$ is defined by
\begin{eqnarray*}
\label{comult}
\vartriangle([a,b]):=\sum_{a\preceq x\preceq b}
[a,x]\otimes [x,b]\,,
\end{eqnarray*}
where local finiteness is essential. One easily checks that these operations are associative and coassociative, respectively.

composition of comultiplication and multiplication $\triangledown\circ\vartriangle: L\rightarrow L$ is called {\bf Hopf square map} (see (\ref{Zsquare}) for the origin of the name). In the present context it is an important combinatorial function, essentially it counts the elements of an interval,
\begin{eqnarray}
\label{count}
\triangledown\circ\vartriangle([a,b]) = |[a,b]|\;[a,b]\,.
\end{eqnarray}

\begin{example}
\label{MatrRepr}
Let $P$ be finite. Then the partial order $\preceq$ can be extended to a total order on $P$ (\cite{Stanley} 3.5), i.e.\ we get an enumeration $P=\{a_1,a_2,\dots ,a_n\}$ of $P$ such that $a_i\preceq a_j\Rightarrow i\le j$. Applying an interval $[a_i,a_j]$ to the matrix $B_{i,j}\in M_n$ (Example \ref{MatrixBialg2}) with entry $1$ at $(i,j)$ and $0$ elsewhere, we see that $L$ can be perceived as subspace of $M_n$, where concatenation $\triangledown$ corresponds to matrix multiplication and $\sum_{k=1}^n [a_k,a_k]\in L$ to the unit matrix $\;U$. The elements of $L$ correspond to upper triangular matrices.

But the just defined comultiplication $\vartriangle_L$ on $L$ does not coincide with the comultiplication $\vartriangle_{M_n}$ of the matrix bialgebra $M_n$ (Example \ref{MatrixBialg3}), we rather get
\begin{eqnarray*}
\vartriangle_L &=& (\pi\otimes \pi)\, \circ \vartriangle_{M_n}\mid_L\,,
\end{eqnarray*}
where $\pi:M_n\rightarrow L$ denotes the natural projection, applying $B_{i,j}$ to $[a_i,a_j]$ if $a_i\preceq a_j$ and to $0$ else. Thus the compatibility condition (\ref{bi1}) for a strong bialgebra is lost. But this loss will be compensated by existence of an antipode as we shall see in Corollary \ref{IncidHopf}.
\end{example}

Before considering the m-weak bialgebra structure on $L$ we are looking for equivalence relations on the base $\intv({\cal P})$, which allow essentially to maintain the above definitions when replacing intervals by their equivalence classes.
An equivalence relation $\sim$ on $\intv({\cal P})$ is called {\bf bialgebra compatible}, if it is
\begin{enumerate}
\item
$\triangledown${\bf -compatible}, i.e.\
$$
[a,b]\sim [a',b'],\; [b,c]\sim [b',c'] \quad \Rightarrow \quad  [a,c]\sim [a',c'] ; $$
\item
$\vartriangle${\bf -compatible}, i.e.\
for $[a,b]\sim [a',b']$ there exists a bijection $f:[a,b]\rightarrow
[a',b']$ such that for all $x\in [a,b]$
$$
[a,x]\sim [a',f(x)] , \qquad [x,b]\sim [f(x),b'] ;
$$
\item
{\bf unitary}, i.e.\ all one point intervals are equivalent,\footnote{\label{FNunitary}For some results it would be sufficient to require that the one point intervals are partitioned in finitely many equivalence classes $I_k$, so that $1_L := \sum_k I_k$ exists. But our main results need a single class.}
$$
[a,a]\sim[b,b] \quad \mbox{ for all } \; a,b \in {\cal P}\,.
$$
\end{enumerate}
In the literature (e.g.\ \cite{Schmitt}) a $\vartriangle$-compatible equivalence relation is called {\it order compatible}.
Notice that $\vartriangle$-compatibility and unitarity of intervals depend only on the intervals themselves, whereas $\triangledown$-compatibility depends on that part of the poset surrounding the intervals.

\begin{example}
On a locally finite poset define $[a,b]\sim [a',b']$ iff the two intervals $([a,b],\preceq)$, $([a',b'],\preceq)$ with the induced ordering are isomorphic posets. This equivalence relation is only $\vartriangle$-com\-patible and unitary but generally not $\triangledown$-com\-patible as the poset with the following Hasse diagram shows.
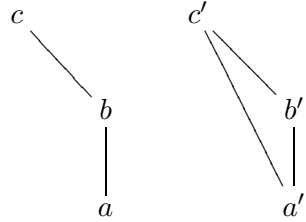
\begin{center}
\begin{minipage}{6cm}
\xymatrix@=2em{
c \ar@{-}[rd] && {c'} \ar@{-}[rdd]\ar@{-}[rd]
 \\
& b \ar@{-}[d] && {b'} \ar@{-}[d]
\\
& a && {a'}
 }
\end{minipage}
\end{center}
\end{example}

Given an equivalence relation $\sim$\, on $\intv({\cal P})$, we denote the equivalence class of an interval $[a,b]$ with $[[a,b]]$ or $[[a,b]]_\sim$. The set
$\intv({\cal P},\sim):=\intv({\cal P})/\sim$ of equivalence classes of intervals is now taken as base for the $K$-linear space $L=L({\cal P},\sim)$.

\begin{proposition}
\label{IntvBialg}
Let ${\cal P}$ be a locally finite poset with a bialgebra compatible equivalence relation $\sim$ on it's intervals.  Then the linear space $L=L({\cal P},\sim)$ becomes a m-weak bialgebra $(L,\triangledown,\eta;\vartriangle,\epsilon)$, if $\;\triangledown$, $\eta$ and $\vartriangle$, $\epsilon$ are defined by
\begin{eqnarray*}
\triangledown(I\otimes J) &:=&
\sum_{x\in [a,b]:\, I=[[a,x]],\, J=[[x,b]]} [[a,b]]\;, \quad \mbox{ with } \sum_\emptyset :=0 \\
1_L &:=& [[a,a]]\,, \qquad \eta: K\hookrightarrow L\,,\quad \eta(r)= r 1_L\,; \\
\vartriangle([[a,b]])&:=& \sum_{x\in [a,b]}
[[a,x]]\otimes [[x,b]]\;, \\
\epsilon([[a,b]]) &:=& \left\{
\begin{array}{ll}
1 & \mbox{if} \quad  a=b \\
0 & \mbox{else}
\end{array} \right., \quad \epsilon: L \rightarrow K\,.
\end{eqnarray*}
If the equivalence relation is only $\triangledown$-compatible and unitary, then $(L,\triangledown,\eta)$ is an algebra.
If the equivalence relation is only $\vartriangle$-compatible, then $(L,\vartriangle,\epsilon)$ is a coalgebra.
\end{proposition}
We will call ${\cal A}({\cal P},\sim) :=(L,\triangledown,\eta)$ the {\bf interval algebra}, ${\cal C}({\cal P},\sim) :=(L,\vartriangle,\epsilon)$ the {\bf interval coalgebra}, and  $(L,\triangledown,\eta;\vartriangle,\epsilon)$ the {\bf (m-weak) interval  bialgebra}.

Multiplication formulates more explicitly as $\triangledown(I\otimes J)=n_{I,J}\cdot [[a,b]]$ if there exist $a,b,x$ such that $I=[[a,x]]$, $J=[[x,b]]$ and $n_{I,J}$ is, for fixed $a, b$, the number of such $x$, and $\triangledown(I\otimes J)=0$ otherwise.
In most of our examples $n_{I,J}=1$ if not zero, the prominent exception is Example \ref{SubsetPoset2}.
 \\

\begin{proof}
The proof is straightforward, since the properties of the equivalence relation had been defined such that $\triangledown$, $\eta$ and $\vartriangle$, $\epsilon$ are well defined, i.e.\ not depending on the intervals representing a class.
\end{proof}

\begin{example}
\label{1equiv}
On a locally finite poset the equivalence relation $$[a,b]\sim_0 [a',b'] \quad :\Leftrightarrow\quad a=b,\; a'=b'$$ is bialgebra compatible.
\end{example}
All unitary, whence all bialgebra compatible equivalence relations are refinements of the relation in Example \ref{1equiv}.

\begin{example}
\label{Antikette0}
Let $\hat{P}=P\cup\hat{0}$ be an antichain $P$ with adjoined minimal element $\hat{0}$. With the bialgebra compatible equivalence relation $\sim_0$ from Example \ref{1equiv} and, with the notations of Proposition \ref{IntvBialg}, the comultiplication becomes
\begin{eqnarray*}
\vartriangle([[a,b]])&=& \left\{
\begin{array}{ll}
1_L\otimes 1_L & \mbox{if} \quad  a=b \\
1_L\otimes [[\hat{0},b]] + [[\hat{0},b]]\otimes 1_L
& \mbox{if} \quad a=\hat{0}\neq b
\end{array} \right. \,.
\end{eqnarray*}
\end{example}

\begin{example}
\label{SubsetPoset}
Let $\Omega$ be an arbitrary set and
${\cal P}_f(\Omega) := \{ A\subseteq\Omega \mid A$ finite\}, then
${\cal P}=({\cal P}_f(\Omega),\subseteq)$ is a locally finite poset.
There is a natural equivalence relation on the set of intervals,
$$
[A,B]\sim_1 [A',B'] \quad :\Leftrightarrow\quad B\setminus A = B'\setminus A' \,.
$$
It is bialgebra compatible and $\;(L({\cal P},\sim_1),\triangledown,\eta;\vartriangle,\epsilon) \;$ is a m-weak bialgebra. Via $[[A,B]]\mapsto B\setminus A$ we identify the base $\intv({\cal P},\sim_1)$ of $L({\cal P},\sim_1)$ with ${\cal P}_f(\Omega)$. Then, for base elements $A,B\in {\cal P}_f(\Omega)$, multiplication is disjoint union,
\begin{eqnarray*}
\triangledown(A\otimes B) = \left\{
\begin{array}{ll}
A\cup B & \mbox{if }\, A\cap B=\emptyset \\
0 & \mbox{else}
\end{array} \right.\,,
\end{eqnarray*}
with $\emptyset$ as unit. Comultiplication is the sum of all decompositions of the set in two sets,
\begin{eqnarray*}
\vartriangle(A) = \sum_{B\subseteq A} B\otimes (A\setminus B)\,,
\end{eqnarray*}
with counit $\varepsilon(A)=\delta_{\emptyset,A}$.

In case of finite $\Omega$, say $\Omega=\{1,\dots,n\}$, we get an algebra isomorphism
\begin{eqnarray*}
{\cal A}({\cal P},\sim_1) & \rightarrow & K[X_1,\dots,X_n]/(X_1^2,\dots,X_n^2) \\
A & \mapsto & \prod_{a\in A} X_a \mod (X_1^2,\dots,X_n^2)\,.
\end{eqnarray*}
of the interval algebra with the algebra of polynomials in $n$ variables modulo the ideal, generated by the squared variables.
\end{example}

\begin{example}
\label{SubsetPoset2}
Regarding the same poset
${\cal P}=({\cal P}_f(\Omega),\subseteq)$ as in Example \ref{SubsetPoset} there is another equivalence relation on the set of intervals,
$$
[A,B]\sim_2 [A',B'] \quad :\Leftrightarrow\quad |B\setminus A| = |B'\setminus A'| .
$$
It is again bialgebra compatible and $\;(L({\cal P},\sim_2),\triangledown,\eta;\vartriangle,\epsilon) \;$ is a m-weak bialgebra. Denote with $I_k$ the class of intervals $[A,B]$, with $|B\setminus A|=k$, then
\begin{eqnarray*}
\triangledown(I_k\otimes I_l) = \left\{
\begin{array}{ll}
{{k+l}\choose k} I_{k+l} & \mbox{if }\, k+l\le|\Omega| \\
0 & \mbox{else}
\end{array} \right. \,, \qquad \vartriangle(I_n) =\sum_{k=0}^n {n\choose k} I_k\otimes I_{n-k}\,.
\end{eqnarray*}
We get an algebra epimorphism
\begin{eqnarray*}
K[X] \rightarrow {\cal A}({\cal P}_f,\sim_2) \,, \quad
X^k \mapsto \left\{
\begin{array}{ll}
k!\, I_k & \mbox{if }\, k\le|\Omega| \\
0 & \mbox{else}
\end{array} \right. \,.
\end{eqnarray*}
For infinite $\Omega$ it is an isomorphism and, if $\Omega$ is finite, the kernel is the ideal $(X^{|\Omega|+1})$.
\end{example}

\begin{example}
\label{Nle}
$(\N_0,\le )$ is the poset of natural numbers
including $0$ with the usual total ordering. It is locally finite. Here we have the natural relation
$$
[a,b]\sim_n [a',b'] \quad :\Leftrightarrow\quad b-a = \,b'- a' \,.
$$
It is bialgebra compatible and $\;(L((\N_0,\le),\sim_n),\triangledown,\eta;\vartriangle,\epsilon) \;$ is a m-weak bialgebra. Denoting with $I_n$ the class of intervals $[a,b]$, with $b-a=n$, then
$$
\triangledown(I_m\otimes I_n) = I_{m+n} \,,\qquad \vartriangle I_n =\sum_{k=0}^n I_k\otimes I_{n-k}\,.
$$
\end{example}

\begin{example}
\label{NDiv}
$(\N,\mid )$ is the set of natural numbers with divisibility as partial order. It is locally finite and the smallest element is $1$. Here again we have a natural relation
$$
[a,b]\sim_n [a',b'] \quad :\Leftrightarrow\quad \frac{b}{a} = \frac{b'}{a'} \,.
$$
It is bialgebra compatible and $\;(L((\N,\mid),\sim_n),\triangledown,\eta;\vartriangle,\epsilon) \;$ is a m-weak bialgebra.
Denoting with $I_n$ the class of intervals $[a,b]$, with $\frac{b}{a}=n$, then
$$
\triangledown(I_m\otimes I_n) = I_{mn} \,,\qquad \vartriangle I_n =\sum_{d\mid n} I_d\otimes I_{\frac{n}{d}}\,.
$$
\end{example}

Let ${\cal P}$ be a locally finite poset with a bialgebra compatible equivalence relation $\sim$. We denote the convolution algebra $\;{\cal H}({\cal C},{\cal A})\;$ of
the  coalgebra $\;{\cal C}=(L({\cal P},\sim),\vartriangle,\epsilon)$ and the algebra  $\;{\cal A}=(L({\cal P},\sim),\triangledown,\eta)$ with
$$
{\cal H}({\cal P},\sim):= (H,\star\,,u)\;,\qquad H=\Hom_K(L,L)\,, \quad L=L({\cal P},\sim)\,,
$$
and call it the {\bf convolution algebra} of $({\cal P},\sim)$.

For later use let us suppose that our poset $\cal{P}$ has a minimal element $\hat{0}$. This had been the case in Examples \ref{Antikette0} through \ref{NDiv}. An {\bf atom} of $\cal{P}$ is an element $a$ such that the interval $[\hat{0},a]$ has exactly two elements. Then for $\varphi, \psi \in H$ preserving the unit
\begin{eqnarray}
\varphi\star\psi([[\hat{0},a]])=\varphi([[\hat{0},a]])+\psi([[\hat{0},a]]) \qquad  &\mbox{if } a \mbox{ is an atom. } 
\end{eqnarray}

\section{The incidence algebra}
\label{Incid}

The incidence algebra of a locally finite poset is the dual of the interval coalgebra. Under a unitarian assumption it comprises the interval algebra as subalgebra and for finite posets both coincide. The incidence algebra is naturally embedded in the convolution algebra of the poset (Theorem \ref{IsubH}) and thus becomes a m-weak Hopf algebra. The antipode is the classical Möbius function.
In this section ${\cal P}=(P,\preceq)$ is a locally finite poset and
$\sim$ a $\vartriangle$-compatible equivalence relation on $\intv({\cal P})$.

In combinatorics the {\bf incidence algebra}
$$
{\cal I}({\cal P},\sim)=(K^{\mbox{\small Intv}({\cal P},\sim)},\star,U)
$$
is defined as the linear space $K^{\mbox{\small Intv}({\cal P},\sim)}=\{\Phi:\intv({\cal P},\sim)\rightarrow K\}$ furnished with
\begin{eqnarray*}
\Phi\star\Psi([[a,b]])&:=& \sum_{a\preceq x\preceq b}
\Phi([[a,x]])\cdot \Psi([[x,b]])\;, \\
U([[a,b]]) &:=& \left\{
\begin{array}{ll}
1 & \mbox{if} \quad  a=b \\
0 & \mbox{else}
\end{array} \right.\,.
\end{eqnarray*}
We write ${\cal I}({\cal P}):={\cal I}({\cal P},=)$ if $\sim$ is the trivial equivalence relation $=$.
The name incidence {\it algebra} is justified in

\begin{proposition}
\label{IncidAlgInv}
Let ${\cal P}=(P,\preceq)$ be a locally finite poset and
$\sim$ a $\vartriangle$-compatible equivalence relation. Then
${\cal I}({\cal P},\sim)$ is an algebra with multiplication $\star$ and unit $U$. An incidence function $\Phi\in{\cal I}({\cal P},\sim)$ is right invertible iff it is left invertible and this holds iff $\Phi([[a,a]])\neq 0$ for all $a\in P$.
\end{proposition}
Multiplication $\star$ is often called {\bf convolution}, see Theorem \ref{IsubH}. \\

\begin{proof}
${\cal I}({\cal P},\sim)$ can be perceived as subspace of ${\cal I}({\cal P})$, where an application $\Phi:\intv({\cal P})\rightarrow K$ belongs to ${\cal I}({\cal P},\sim)$ iff it is constant on the equivalence classes of $\sim$. Therefore it is sufficient to prove the case ${\cal I}({\cal P})$, what is done in e.g.\ \cite{Stanley} 3.6.2.
\end{proof}

We expand the domain of the incidence functions $\Phi\in {\cal I}({\cal P},\sim)$ to the linear space $L({\cal P},\sim)$ with base $\intv({\cal P},\sim)\;$,
\begin{eqnarray*}
\Phi &\mapsto & \overline{\Phi}:L({\cal P},\sim)\rightarrow K\quad \mbox{ with }\; \overline{\Phi}|_{\mbox{\small Intv}({\cal P},\sim)} =\Phi \,,
\end{eqnarray*}
i.e.\ $\overline{\Phi}$ denotes the linear form, coinciding on the base $\intv({\cal P},\sim)\;$ with $\Phi$.

\begin{proposition}
\label{IncCStern}
The incidence algebra ${\cal I}({\cal P},\sim)$ is canonically isomorphic to the dual ${\cal C}^*$ of the interval coalgebra ${\cal C}:=(L({\cal P},\sim),\vartriangle,\epsilon)$\,,
\begin{eqnarray*}
{\cal I}({\cal P},\sim) \;\;\rightarrow\; {\cal
C}^* \,,\qquad \Phi \mapsto \overline{\Phi}\,.
\end{eqnarray*}
\end{proposition}

\begin{proof}
Clearly $\overline{U}=\epsilon$, the counit of ${\cal C}$ and the unit of ${\cal C}^*$ (see Example \ref{CoalgDual}). For the multiplication the following commuting diagram proves the desired equation.
\begin{center}
\begin{minipage}{6cm}
\xymatrix@=5em{
{L}
\ar[rrd]^{\quad\overline{\Phi\star\Psi} = \overline{\Phi}\star\overline{\Psi}}
\ar[ddr]_{\vartriangle}
 \\
& {\intv({\cal P},\sim)}
\ar[r]_{\Phi\star\Psi}
\ar[d]^{\vartriangle} \ar[ul] & K
 \\
& {L \otimes L}
\ar[r]^{\overline{\Phi}\otimes\overline{\Psi}}
& {K \otimes K}
\ar[u]^{\triangledown_K}
 }
\end{minipage}
\end{center}
Here $L$ abbreviates $L({\cal P},\sim)$ as usual and $\intv({\cal P},\sim)\hookrightarrow L({\cal P},\sim)$ is the embedding of the base in it's span.
\end{proof}
\noindent
In the sequel we identify ${\cal I}({\cal P},\sim)$ with ${\cal C}^*$ and perceive $\Phi\in {\cal I}({\cal P},\sim)$ as application $L\rightarrow K$.

\begin{proposition}
\label{IntvIncidAlg}
Let ${\cal P}=(P,\preceq)$ be a locally finite poset with a bialgebra compatible equivalence relation $\sim$ on it's intervals. Then the interval algebra ${\cal A}({\cal P},\sim)=(L({\cal P},\sim),\triangledown,\eta)$ is a subalgebra of the incidence algebra ${\cal I}({\cal P},\sim)$.

Furthermore, ${\cal A}({\cal P},\sim)={\cal I}({\cal P},\sim)$ iff $P$ is finite.
\end{proposition}

\begin{proof}
Clearly $L({\cal P},\sim)$ is linearly embedded in $K^{\mbox{\small Intv}({\cal P},\sim)}$ by means of
$$
\alpha:=\sum_{I\in{\mbox{\small Intv}({\cal P},\sim)}} \alpha_I \,I \;\mapsto\; (I\mapsto \alpha_I)
$$
Denote with $\Phi_\alpha$ the image of $\alpha$ under this embedding, then $\Phi_I(J)=\delta_{I,J}$. We have to show that
\begin{eqnarray}
\label{1star}
\Phi_I\star\Phi_J = \Phi_{\triangledown(I\otimes J)}\,,\qquad \Phi_{[[a,a]]}=U\,.
\end{eqnarray}
The second equation being obvious, we check the first, using Proposition \ref{IntvBialg}.
\begin{eqnarray*}
\Phi_I\star\Phi_J ([[a,b]]) 
&=& \sum_{x\in [a,b]} \Phi_I([[a,x]]) \Phi_J ([[x,b]]) \\
&=& |\{x\in [a,b] \mid [[a,x]]=I,\; [[x,b]]=J \}| \\
&=& \Phi_{\triangledown(I\otimes J)}([[a,b]])\,.
\end{eqnarray*}
Since any $\alpha \in L({\cal P})$ is a linear combination of finitely many intervals the last assertion is clear.
\end{proof}

\noindent
Proposition \ref{IntvIncidAlg} holds, too, with finite $P$ and the trivial equivalence relation $=$, if one modifies the second equation in (\ref{1star}) to $\Phi_{\sum_{a\in P}[a,a]}=U$, see footnote \ref{FNunitary}.

Often, the incidence algebra can be embedded in the convolution algebra. Recall from end of Section \ref{Poset} that the latter is defined only for bialgebra compatible equivalence relations.

\begin{theorem}
\label{IsubH}
Let ${\cal P}=(P,\preceq)$ be a locally finite poset and
$\sim$ a bialgebra compatible equivalence relation on $\intv({\cal P})$ with the property, that the product of intervals is again an interval or zero, equivalently (cf.\ Proposition \ref{IntvBialg})
\begin{eqnarray}
\label{1xinIntv}
\triangledown([[a,x]]\otimes [[x,b]]) &=& [[a,b]] \quad\mbox{ for } a\preceq x\preceq b\,.
\end{eqnarray}
For $\Phi\in {\cal I}({\cal P},\sim)$ let $\widehat{\Phi}\in {\cal H}({\cal P},\sim)$ be the linear application given on the base $\intv({\cal P},\sim)$ by
\begin{eqnarray*}
\widehat{\Phi}(I) := \Phi(I) I\,,\qquad I \in \intv({\cal P},\sim)\,.
\end{eqnarray*}
Then the application
\begin{eqnarray*}
{\cal I}({\cal P},\sim) \;\;\hookrightarrow\; {\cal
H}({\cal P},\sim) \,,\quad \Phi \mapsto \widehat{\Phi}\,
\end{eqnarray*}
is a monomorphism of algebras. The image $\widehat{{\cal I}}({\cal P},\sim)$ of ${\cal I}({\cal P},\sim)$ is a subalgebra of ${\cal
H}({\cal P},\sim)$.

Furthermore $f\in \widehat{{\cal I}}({\cal P},\sim)$ respects the unit, $f\circ \eta=\eta$, iff it respects the counit, $\epsilon\circ f=\epsilon$. And this holds iff $f$ has a convolution inverse $f^{\star -1}\in \widehat{{\cal I}}({\cal P},\sim)$.

\end{theorem}
$\widehat{\Phi}$ applies an interval class $I$ to a scalar
multiple of this interval class, whence regarding $\widehat{\Phi}$
as linear operator, it's spectrum consists of the values of $\Phi:\intv({\cal P},\sim)\rightarrow K$
and the interval classes form a base of
$\widehat{\Phi}$-eigenvectors. \\

\begin{proof}
The application $\;\widehat{ }\;$ clearly is linear. It preserves the unit, $\;\widehat{U}=\eta\circ\epsilon=u \;$ since $\sim$ is unitary. Finally $\;\widehat{ }\;$ is compatible with convolution. It is sufficient to check this on the base elements $[[a,b]]\in \intv({\cal P},\sim)$. Using the definitions of convolution in Theorem \ref{convAlg} and of $\vartriangle$, $\triangledown$ in Proposition \ref{IntvBialg} we get
\begin{eqnarray*}
\widehat{\Phi} \star \widehat{\Psi}([[a,b]])&=& \triangledown\circ(\widehat{\Phi} \otimes \widehat{\Psi})\circ \vartriangle ([[a,b]]) \\
&=& \sum_{x\in[a,b]}
\triangledown(\widehat{\Phi}([[a,x]])\,\otimes\widehat{\Psi}([[x,b]])) \\
&=& \sum_{x\in[a,b]}
\Phi(a,x) \Psi(x,b)\;\triangledown([[a,x]]\otimes[[x,b]]) \\
&=& \Phi\star\Psi(a,b)\;[[a,b]] \\
&=& \widehat{\Phi\star\Psi}([[a,b]])\,.
\end{eqnarray*}
For the fourth equation we used the definition of $\star$ in ${\cal I}({\cal P},\sim)$ and property (\ref{1xinIntv}) of $\triangledown$. The reader should draw a diagram like in the proof of Proposition \ref{ConvMorph}.
\end{proof}

The constant function $1$ on $\intv({\cal P},\sim)$ is called the {\bf zeta function} of $({\cal P},\sim)$ and denoted $\Zeta$. It is invertible (Proposition \ref{IncidAlgInv}) and it's inverse $\Mu:=\Zeta^{\star -1}$ is called the {\bf Möbius function}. These names had been given (see \cite{Rota}) in analogy to the classical case ${\cal P}=(\N,\mid)$ (see Example \ref{Nteilt}).

The image $\widehat{\Zeta}$ of the zeta function in the convolution algebra is the identity map $\id_L$ and
\begin{eqnarray}
\label{Zsquare}
\widehat{\Zeta} \star \widehat{\Zeta} = \triangledown\circ(\id_L\otimes \id_L)\circ \vartriangle =\triangledown\circ\vartriangle
\end{eqnarray}
is the Hopf square map (see (\ref{count})).

\begin{corollary}
\label{IncidHopf}
Under the assumptions of Theorem \ref{IsubH}\, the m-weak interval bialgebra $(L({\cal P},\sim),\triangledown,\eta;\vartriangle, \epsilon)$ (see Proposition \ref{IntvBialg}) is a m-weak Hopf algebra. The antipode is $S = \widehat{\Mu}$.
\end{corollary}
Especially, the incidence algebra of a finite Poset ${\cal P}=(P,\preceq)$ with property (\ref{1xinIntv}) is a m-weak Hopf algebra with antipode $S = \widehat{\Mu}$. \\

\begin{proof}
$\Mu=\Zeta^{\star-1}$ and $\widehat{\Zeta}=\id_L$ imply $\widehat{\Mu}=\widehat{\Zeta^{\star-1}}=\widehat{\Zeta}^{\star-1}=\id_L^{\star-1}$ and this is the antipode $S$.
\end{proof}

$S$ is a bialgebra morphism iff
$$S\circ\triangledown = \triangledown\circ (S\otimes S)\,,\quad \vartriangle \circ S = (S\otimes S)\circ \vartriangle\,,$$
i.e.\ iff the following diagram commutes.
\begin{center}
\begin{minipage}{6cm}
\xymatrix@=3em{
{L}
\ar[r]^{S} \ar@/_ 1cm/[dd]_{\widehat{\Zeta}^{\star 2}}
\ar[d]^{\vartriangle}
& {L}
\ar[d]^{\vartriangle} \ar@/^ 1cm/[dd]^{\widehat{\Zeta}^{\star 2}}
 \\
{L\otimes L} \ar[d]^{\triangledown}
\ar[r]^{S\otimes S}
& {L\otimes L}  \ar[d]^{\triangledown}
 \\
L
\ar[r]^{S}
& L
 }
\end{minipage}
\end{center}
In general this is not the case but all paths from north west to south east give the same application,
\begin{eqnarray*}
S\circ\triangledown\circ\vartriangle = \triangledown\circ(S\otimes S)\circ\vartriangle = S \star S =          \triangledown\circ\vartriangle\circ S\,.
\end{eqnarray*}

The linear space $K^{\mbox{\small Intv}({\cal P},\sim)}$ which underlies the incidence algebra ${\cal I}({\cal P},\sim) = (K^{\mbox{\small Intv}({\cal P},\sim)},\star,U)$ supports a second algebra ${\cal A}_\circ({\cal P},\sim) := (K^{\mbox{\small Intv}({\cal P},\sim)},\cdot\,,\Zeta)$ with $\cdot$ denoting usual multiplication of functions and the constant function $\Zeta\equiv 1$ being the unit. The application $\;\widehat{ }\;$ transforms $\Zeta$ in $\id_L$ and the product to composition,
$$
\widehat{\Phi\cdot\Psi} = \widehat{\Phi}\circ\widehat{\Psi}\,.
$$
The image\footnote{the underlying linear spaces of $\widehat{\cal A_\circ}({\cal P},\sim)$ and $\widehat{\cal I}({\cal P},\sim)$ are identical.}  $\;\widehat{\cal A_\circ}({\cal P},\sim)\subseteq H=\Hom_K(L,L)\,$, $L=L({\cal P},\sim)$, of this algebra is a subalgebra of the composition algebra $(H,\circ, \id_L)$. Even more it becomes a bialgebra, but not canonically.

\begin{proposition}
\label{compBialg}
Let ${\cal P}=(P,\preceq)$ be a locally finite poset and
$\sim$ a bialgebra compatible equivalence relation on $\intv({\cal P})$.
There exists a unitary base ${\cal B}$ of $\widehat{\cal A_\circ}({\cal P},\sim)$ and for any such base
$(\widehat{\cal A_\circ}({\cal P},\sim),\circ\,,\id_L;\vartriangle_{{\cal B}},\epsilon_{{\cal B}})$ is a bialgebra, where
\begin{eqnarray*}
\vartriangle_{{\cal B}}(\widehat{b}) &=& \widehat{b}\otimes \widehat{b}\;,\qquad \widehat{b}\in {\cal B}\,, \\
\epsilon_{\cal B}(\widehat{\Phi}) &=& \Phi([[a,a]])\;,\qquad \Phi\in K^{\mbox{\small Intv}({\cal P},\sim)}\;.
\end{eqnarray*}
\end{proposition}

\begin{proof}
From Example \ref{AlgDual} we know that there exists a unitary base ${\cal B}_0$ of $L^*$, that is $b\circ \eta(1_K) =1_K$ for all $b\in{\cal B}_0$. Then ${\cal B}:=\{\widehat{b}\mid b\in{\cal B}_0 \}$ is a base of $\widehat{\cal A}_\circ({\cal P},\sim)$ and $\widehat{b}\,([[a,a]])=b\,([[a,a]])\cdot [[a,a]]= [[a,a]]$ for all $\widehat{b}\in {\cal B}$. The assumptions of Proposition \ref{composBialg} are verified and we are done.
\end{proof}

In this context the following result is remarkable.

\begin{proposition}
\label{charIncidSubalg}
Let ${\cal P}=(P,\preceq)$ be a finite poset and $\cal{A}\subseteq {\cal I}({\cal P})$ a subalgebra of the incidence algebra. Then $\widehat{\cal A}\subseteq \;\widehat{\cal A_\circ}({\cal P},=)$ is a subalgebra of  the composition algebra iff there exists a $\vartriangle$-compatible equivalence relation $\sim$ on $\intv({\cal P})$ such that $\cal{A}= {\cal I}({\cal P},\sim)$.
\end{proposition}

\begin{proof}
See e.g. \cite{SpiegelD} Proposition 1.3.9.
\end{proof}

\section{Examples and incidence bialgebra}
\label{Examples}

Proposition \ref{compBialg} shows, that the dual of the algebra $(L,\triangledown,\eta)$ is not the appropriate coalgebra (see Example \ref{AlgDual}) in order to make the incidence algebra a bialgebra. Rather, like in the case of finite dimensions (Proposition \ref{IntvIncidAlg}),  one has to use the comultiplcation $\vartriangle$ on $L$ to define the suitable comultiplication on the incidence algebra. The next examples will pave the way. Notice that all examples presented here are commutative and cocommutative.

\begin{example}
\label{SubsetPosetPoly}
Let $K=\C$. We resume Example \ref{SubsetPoset2} of the poset ${\cal P}=({\cal P}_f(\Omega),\subseteq)$ with countable $\Omega$ and the relation $\sim_2$. We denoted with $I_n$ the interval class $[[A,B]]$, with $|B\setminus A|=n$ and have seen, that
\begin{eqnarray*}
\;(L({\cal P},\sim_2),\triangledown,\eta;\vartriangle,\epsilon) &\rightarrow& (\C[X],\cdot\,,1;\vartriangle_2,\epsilon_2) \\
I_n &\mapsto& \frac{1}{n!}\, X^n
\end{eqnarray*}
is an algebra isomorphism. One easily checks that it is a coalgebra morphism, too, whence a bialgebra morphism. So the interval bialgebra is a Hopf algebra with antipode $S(I_n)=(-1)^n I_n$ (Example \ref{Polynbialg}).
This isomorphism, Proposition \ref{IncCStern} and Example \ref{PowerSeries} give us the algebra isomorphism
\begin{eqnarray*}
{\cal I}({\cal P},\sim_2) &\rightarrow& ((\C[X],\vartriangle_2,\epsilon_2)^*,\star_2\,,1) \simeq (\C[[X]],\cdot\,,1) \\
\Phi &\mapsto& \widetilde{\Phi}\;:=\;\sum_{n=0}^\infty \frac{\Phi(I_n)}{n!} X^n
\end{eqnarray*}
Especially $\widetilde{\Zeta}=e^X$ is the exponential function and  $\widetilde{\Mu}=e^{-X}$ it's inverse function. Then the Möbius function is $\Mu(I_n)=(-1)^n$.

We regard another remarkable incidence function, $\Phi(I_n)=\frac{1}{n+1}$ (see \cite{Denneberg and Grabisch}). It's image $\widehat{\Phi}\in {\cal H}({\cal P},\sim_2)$ is neither an algebra morphism nor a coalgebra morphism. One easily computes $X \widetilde{\Phi} = \widetilde{\Zeta} -1$, so that $\widetilde{\Phi}^{-1}$ is the exponential generating function
$\frac{X}{e^X -1}=\;\sum_{n=0}^\infty \frac{\beta_n}{n!} X^n$
of the Bernoulli numbers $(\beta_n)_{n\in\N_0}$, whence $\Phi^{\star-1}(I_n)=\beta_n$.

The comultiplication $\vartriangle_2$ and the counit $\epsilon_2$ of the bialgebra of polynomials can be extended to the algebra of formal power series,
$$
\vartriangle_2(f)(X,Y) = f(X+Y)\;,\quad \epsilon_2(f)=f(0) \qquad \mbox{for } f\in \C[[X]]\,,
$$
but the question is, if the power series $\vartriangle_2(f)$ in two variables is contained in $\,\C[[X]]\otimes \C[[X]]$, which is only a subalgebra of $\C[[X,Y]]$.
The answer is affirmative:

The linear space $\C[[X]]$ is the directed colimit\footnote{{\it Directed colimits} are often called {\it direct limit} or {\it inductive limit}.} of the directed poset (with set inclusion as ordering) of the linear spaces $D_n$, $n\in\N_0$, of polynomials of degree $\le n$. Since the functor $M\mapsto M\otimes M$ on the category of $K$-linear spaces preserves directed colimits (\cite{Porst} Proposition 8), $\C[[X]]\otimes \C[[X]]$ is the colimit of $(D_n\otimes D_n)_{n\in\N_0}$. The universal property of directed colimits implies that the linear applications $\vartriangle_2: D_n\rightarrow D_n\otimes D_n$, $n\in \N_0$, have a directed colimit. Use the $(X)$-adic and $(X,Y)$-adic topologies to show that this application coincides with $\vartriangle_2: \C[[X]]\rightarrow \C[[X,Y]]\supset \C[[X]]\otimes \C[[X]]$ as defined above.

$(\C[[X]],\cdot\,,1;\vartriangle_2,\epsilon_2)$ is a bialgebra and like in Examples \ref{Polynbialg} and \ref{Polynbialg2} one verifies that it is a Hopf algebra with antipode $S_2$, which is the bialgebra morphism definded by $S_2(f)(X)=f(-X)$. Then $S_2(\widetilde{\Zeta})=\widetilde{\Mu}$.

Summarizing, in the present example the incidence algebra is a Hopf algebra and the antipode maps the zeta function to the Möbius function and vice versa.
\end{example}

\begin{example}
\label{Nklgl}
We resume Example \ref{Nle} of the poset ${\cal P}=(\N_0,\le)$ with the relation $\sim_n$. Denote with $I_n$ the interval class $[[a,b]]$, with $|b-a|=n$. Then one easily checks that
\begin{eqnarray*}
(L((\N,\le),\sim_n),\triangledown,\eta;\vartriangle,\epsilon) &\rightarrow& (\C[X],\cdot\,,1;\vartriangle_1,\epsilon_1) \\
I_n &\mapsto& X^n
\end{eqnarray*}
is a bialgebra isomorphism. So the interval bialgebra is a m-weak Hopf algebra with antipode $S$, $S(I_0)=1$, $S(I_1)=-1$ and $S(I_n)=0$ for $n>1$ (Example \ref{Polynbialg}), which is the Möbius function $\widehat{\Mu}$ by Corollary \ref{IncidHopf}.

This isomorphism induces the algebra isomorphism
\begin{eqnarray*}
{\cal I}((\N,\le),\sim_n) &\rightarrow& ((\C[X],\vartriangle_1,\epsilon_1)^*,\star_1\,,1) \simeq (\C[[X]],\cdot\,,1) \\
\Phi &\mapsto& \widetilde{\Phi}\;:=\;\sum_{n=0}^\infty \Phi(I_n) X^n
\end{eqnarray*}
Especially $\widetilde{\Zeta}=\sum_{n=0}^\infty X^n$ is the geometric series and  $\widetilde{\Mu}=1-X$ it's inverse.

Like in the last example, $\vartriangle_1$ and $\epsilon_1$ can be extended from polynomials to power series and we get a m-weak bialgebra $(\C[[X]],\cdot\,,1;\vartriangle_1,\epsilon_1)$. It is a m-weak Hopf algebra with antipode $S_1$, definded, like for the polynomial m-weak Hopf bialgebra $(\C[X],\cdot\,,1;\vartriangle_1,\epsilon_1)$, by
$S_1\left(\sum_{n=0}^\infty \varphi_n X^n\right):= \varphi_0 - \varphi_1 X$.
Again $S_1(\widetilde{\Zeta})=\widetilde{\Mu}$, but $S_1\circ S_1=\id$ fails (cf. Corrollary \ref{ScircS}).
\end{example}

\begin{example}
\label{Nteilt}
We resume Example \ref{NDiv} of the m-weak interval bialgebra $\;(L((\N,|\,),\sim_n),\triangledown,\eta; \vartriangle, \epsilon) \;$ of the divisibility order with the natural equivalence relation $\sim_n$.

Here we need the $\C$-linear space $\C\langle s\rangle$ spanned by the linearly independent functions $(b_n)_{n\in\N}$,
$$
b_n : \C \rightarrow \C\,, \qquad b_n(s):=n^{-s}\,, s\in\C\,.
$$
Notice that
$$
b_m\cdot b_n=b_{mn}\;, \mbox{ i.e.}\quad \triangledown(b_m\otimes b_n)=b_{mn}\,.
$$
$\C\langle s\rangle$ is a subalgebra of the algebra $\C\langle\langle s\rangle\rangle$ of formal {\bf Dirichlet series}
\begin{eqnarray*}
\varphi(s) &=& \sum_{n=1}^\infty \varphi_n n^{-s}\,,\qquad \varphi_n \in\C\,.
\end{eqnarray*}
We call the functions in $\C\langle s\rangle$ {\bf Dirichlet polynomials}. Like for power series there is an algebra monomorphism of $\C\langle\langle s\rangle\rangle \otimes \C\langle\langle s\rangle\rangle$ to the algebra $\C\langle\langle s,t\rangle\rangle$ of formal Dirichlet series in two variables,
\begin{eqnarray}
\label{2Variable}
\C\langle\langle s\rangle\rangle \otimes \C\langle\langle s\rangle\rangle \hookrightarrow \C\langle\langle s,t\rangle\rangle \,,\qquad m^{-s} \otimes n^{-s} \mapsto m^{-s}n^{-t}\,.
\end{eqnarray}

Now we get algebra isomorphisms
\begin{eqnarray*}
(L((\N,|\,),\sim_n),\triangledown,\eta) &\rightarrow & (\C\langle s \rangle,\cdot\,,1)\;, \qquad [[1,n]] \mapsto b_n\,,  \\
{\cal I}((\N,|\,),\sim_n) &\rightarrow & (\C\langle\langle s\rangle\rangle,\cdot\,,1)\;, \qquad \Phi \mapsto \widetilde{\Phi}:=\sum_{n=1}^\infty \Phi([[1,n]])\, n^{-s}\;.
\end{eqnarray*}
The image of $\Zeta \in {\cal I}((\N,|\,),\sim_n)$ is the Riemann zeta function $\zeta(s)=\sum_{n=1}^\infty n^{-s}$ and the image of $\Mu=\Zeta^{\star-1}$ is the Dirichlet series $\frac{1}{\zeta(s)}=\sum_{n=1}^\infty \mu_n\, n^{-s}$, where the coefficients are the values of the classical Möbius function,
$$
\mu_n = \left\{ \begin{array}{l@{\quad}l} (-1)^k & \mbox{
if } n=p_1p_2 \dots p_k \mbox{ with primes } p_1 < p_2 < \dots <
p_k
\\
0 & \mbox{ else}
\end{array}
\right..
$$

Like in the preceding examples the coalgebra structure
\begin{eqnarray*}
\vartriangle(b_n) &=& \sum_{d|n} b_d\otimes b_{\frac{n}{d}} \,,\qquad
\epsilon(b_n) \;:=\; \left\{
\begin{array}{ll}
1 & \mbox{if} \quad  n=1 \\
0 & \mbox{else}
\end{array} \right.\,
\end{eqnarray*}
on $L((\N,|\,),\sim_n)$ transports to $\C\langle s\rangle$ and extends \, (using $\emph{m}$-adic topologies, the maximal ideal $\emph{m}$ of $\C\langle\langle s\rangle\rangle$ is generated by $b_2,b_3\dots$)\; to $\C\langle\langle s\rangle\rangle$.
Using (\ref{2Variable}), $\vartriangle$ can be written as
\begin{eqnarray*}
\vartriangle(b_n)(s,t) &=& \sum_{d|n} d^{-s} (\frac{n}{d})^{-t} \;=\; n^{-t} \sum_{d|n} d^{-(s-t)}\,.
\end{eqnarray*}
Especially $\vartriangle(\zeta)(s,t)=\zeta(s)\cdot\zeta(t)$ and $\vartriangle(\zeta^{-1})(s,t)=\zeta^{-1}(s)\cdot\zeta^{-1}(t)$. This holds despite the fact, that we have only a m-weak bialgebra $(\C\langle\langle s\rangle\rangle,\cdot\,,1;\vartriangle,\epsilon)$.

Setting $t=s$ we get $\;\vartriangle(n^{-s})(s,s)=d_n\, n^{-s}$ with $d_n$ denoting the number of divisors of $n$ (cf.\ (\ref{count})).

Finally we look at the Hopf algebra structure of the m-weak interval bialgebra and the  m-weak incidence bialgebra of the divisibility ordering w.r.t.\ $\sim_n$.
The m-weak interval bialgebra $(L((\N,|\,),\sim_n),\triangledown,\eta;
\vartriangle, \epsilon)$ is a m-weak Hopf algebra by Corollary \ref{IncidHopf} with antipode $\widehat{\Mu}$, where $\Mu([[1,n]])=\mu_n$ for $n\ge 1$.
The m-weak bialgebra $(\C\langle\langle s\rangle\rangle,\cdot\,,1; \vartriangle, \epsilon)$ has antipode $S$, given by $S(\varphi)= \sum_{n=1}^\infty \varphi_n \mu_n b_n$. The proof is straightforward, verifying $\id\star S=u$ and $S\star \id=u$ using
\begin{eqnarray*}
\sum_{d|n} \mu_d \;:=\; \left\{
\begin{array}{ll}
1 & \mbox{if} \quad  n=1 \\
0 & \mbox{else}
\end{array} \right.\,.
\end{eqnarray*}
This is just the equation $\Zeta\star \Mu = U$ in the incidence algebra of the subposet $D_n:=\{d\in \N \mid d|n\}\;$ of $(\N,\mid\,)$ with the equivalence relation $\sim_n$ restricted to $D_n$. On $\intv((D_n,\mid\,),\sim_n)$ the zeta and Möbius functions are the restrictions of the zeta and Möbius functions on $\intv((\N,\mid\,),\sim_n)$ (Proposition \ref{IncMorph}).
\end{example}

In all these examples the incidence algebra is a function space and the comultiplicaton, introduced on it, transforms a function of one variable in a function of two variables. In the last example we have seen, that inserting a specific value for the new covariable can reveal an interesting combinatorial function of the respective poset. Here is another example.

\begin{example}
Inserting in Example \ref{SubsetPosetPoly} $\;Y=-a \in \C$ in the coproduct of $f\in \C[[X]]$ gives the power series expansion $\vartriangle_2(f)(X,-a)=f(X-a)$ of the incidence function $f$ at the point $a$.

Setting $Y=X$ gives $\vartriangle_2(X^n)(X,X)=2^n\,X^n$. Here $2^n$ is the number of elements in any interval of the class $I_n$ (cf.\ (\ref{count})).
\end{example}

How to construct generally a comultiplication on $\infty$-dimensional incidence algebras? The examples suggest to extend the finite dimensional case via directed colimits. We leave this to be elaborated.

\section{Morphisms of incidence algebras}

Given two locally finite posets ${\cal P}_i=(P_i,\preceq_i)$ with
$\vartriangle$-compatible equivalence relations $\sim_i$ on $\intv({\cal P}_i)$, $i=1$, $2$. A linear application
$$\gamma: L_1 \rightarrow L_2$$
 between the corresponding interval spaces $L_i:=L({\cal P}_i,\sim_i)$ induces by dualisation a linear morphism between the incidence algebras ${\cal I}_i:={\cal I}({\cal P}_i,\sim_i)$,
\begin{eqnarray*}
\gamma^* : {\cal I}_2 &\rightarrow& {\cal I}_1\;,\qquad \Phi_2 \mapsto \Phi_2 \circ \gamma\;.
\end{eqnarray*}
If $\gamma$ is injective (surjective), then $\gamma^*$ is surjective (injective).
The question is, if $\gamma^*$ is an algebra morphism.

\begin{proposition}
\label{IncMorph}
Given two locally finite posets ${\cal P}_i=(P_i,\preceq_i)$ with
$\vartriangle$-compatible equivalence relations $\sim_i$ on $\intv({\cal P}_i)$, $i=1$, $2$, and a coalgebra morphism $\gamma: {\cal C}_1 \rightarrow {\cal C}_2$ between the interval coalgebras ${\cal C}_i=(L({\cal P}_i,\sim_i),\vartriangle_i,\epsilon_i)$, then $\gamma^*$ is an algebra morphism.
\end{proposition}

\begin{proof}
Applying Proposition \ref{ConvMorph} with $\alpha=\id_K$ we get an algebra morphism ${\cal C}_2^* \rightarrow {\cal C}_1^*$. By Proposition \ref{IncCStern} we are done.
\end{proof}

\begin{example}
\label{RefinemMorph}
Let ${\cal P}$ be a locally finite poset with two bialgebra compatible equivalence relations $\sim_1$ and $\sim_2$, the latter being a refinement of the first. The natural projection
$
\intv({\cal P},\sim_1) \rightarrow \intv({\cal P},\sim_2)
$
induces a linear application
$$\gamma: L({\cal P},\sim_1) \rightarrow L({\cal P},\sim_2)\,,\quad [[a,b]]_{\sim_1}\mapsto[[a,b]]_{\sim_2}$$
This is a coalgebra morphism so that by Proposition \ref{IncMorph}
\begin{eqnarray*}
\gamma^* : {\cal I}({\cal P},\sim_2) &\rightarrow& {\cal I}({\cal P},\sim_1) \,,\quad \Phi\mapsto \Phi\circ\gamma
\end{eqnarray*}
is an algebra morphism. It is a monomorphism since $\gamma$ is surjective.
\end{example}

\begin{example}
\label{Power2-1}
Regard Example \ref{RefinemMorph} with the poset $({\cal P}(\{1,\dots,n\},\subseteq)$ and the equivalence relations $\sim_1$ (Example \ref{SubsetPoset}) and $\sim_2$ (Example \ref{SubsetPoset2}). Then we get
\begin{eqnarray*}
\gamma^* : \C[X]/(X^{n+1}) & \hookrightarrow & \C[X_1,\dots,X_n]/(X_1^2,\dots,X_n^2) \\
X & \mapsto & X_1 + \dots + X_n \mod (X_1^2,\dots,X_n^2)\,.
\end{eqnarray*}
\end{example}

\begin{example}
We denote with $\Prim$ the set of prime numbers. Let ${\cal P}_1$ be the poset $({\cal P}_f(\Prim),\subseteq)$ with the relation $\sim_1$  (Example \ref{SubsetPoset}) and ${\cal P}_2 = (\N,|)$ with $\sim_n$  (Example \ref{Nteilt}). They are closely related through the poset embedding
\begin{eqnarray}
\label{PfN} T: ({\cal P}_f(\Prim),\subseteq) \hookrightarrow
(\N,|)\,, \quad P \mapsto \prod_{p \in P} p\,, \quad \emptyset
\mapsto 1\,,
\end{eqnarray}
assigning to any finite set $P \subset
\Prim$ of primes their product. The image of this application is
$$
Q:=\{n \in {\N} \mid p^2\not|\, n \mbox{ for
all } p\in{\Prim}\} =\{1,2,3,5,6,7,10,11,13, \dots \}\,,
$$
the {\bf set of squarefree numbers}. $({\cal P}_f(\Prim), \subseteq)$ and
$(Q,|)$ are canonically isomorphic posets.
Furthermore $(Q,|)$ is a subposet of $(\N,|)$ with the property
that any interval $[x,y]$ in $Q$ is also an interval in $\N$ since
a divisor of a squarefree number is itself squarefree. Also the equivalence relations $\sim_1$ on $({\cal P}_f(\Prim),\subseteq)$ (Example \ref{SubsetPoset}) and $\sim_n$ on $(Q,|)$ (Example \ref{NDiv}) are the same modulo the isomorphism. Then the injection
$$
\intv(({\cal P}_f(\Prim),\subseteq),\sim_1) \hookrightarrow \intv((\N,|),\sim_n) \,.
$$
induces a linear application between the spans of these interval sets,
$$
\gamma: L(({\cal P}_f(\Prim),\subseteq),\sim_1) \hookrightarrow L((\N,|),\sim_n)\;.
$$
Notice that this is no algebra morphism, since the domain contains zero divisors (see Example \ref{SubsetPoset}) and the range not. But $\gamma$ is a coalgebra morphism and by Proposition \ref{IncMorph} we get an algebra morphism
\begin{eqnarray*}
\gamma^* : {\cal I}((\N,|),\sim_n) &\rightarrow& {\cal I}(({\cal P}_f(\Prim),\subseteq),\sim_1) \simeq {\cal I}((Q,|),\sim_n) \,,\quad \Phi\mapsto\Phi|_Q\,.
\end{eqnarray*}
It is an epimorphism since $\gamma$ is injective. Furthermore $\gamma^* \circ\gamma=\id$.
\end{example}


\end{document}